\documentclass[oneside,english]{amsart}
\usepackage[T1]{fontenc}
\usepackage[latin9]{inputenc}
\usepackage{textcomp}
\usepackage{amsthm}
\usepackage{amstext}
\usepackage{esint}

\makeatletter
\numberwithin{equation}{section}
\numberwithin{figure}{section}
\theoremstyle{plain}
\newtheorem{thm}{\protect\theoremname}[section]
  \theoremstyle{plain}
  \newtheorem{cor}[thm]{\protect\corollaryname}
  \theoremstyle{definition}
  \newtheorem{example}[thm]{\protect\examplename}
  \theoremstyle{plain}
  \newtheorem{lem}[thm]{\protect\lemmaname}
  \theoremstyle{remark}
  \newtheorem{rem}[thm]{\protect\remarkname}
  \theoremstyle{plain}
  \newtheorem{prop}[thm]{\protect\propositionname}

\def\makebbb#1{
    \expandafter\gdef\csname#1\endcsname{
        \ensuremath{\Bbb{#1}}}
}\makebbb{R}\makebbb{N}\makebbb{Z}\makebbb{C}\makebbb{H}\makebbb{E}\makebbb{H}\makebbb{P}\makebbb{B}\makebbb{Q}\makebbb{E}

\usepackage{babel}

\usepackage{babel}

\makeatother

\usepackage{babel}

\makeatother

\usepackage{babel}
  \providecommand{\corollaryname}{Corollary}
  \providecommand{\examplename}{Example}
  \providecommand{\lemmaname}{Lemma}
  \providecommand{\propositionname}{Proposition}
  \providecommand{\remarkname}{Remark}
\providecommand{\theoremname}{Theorem}

\begin{document}

\title{K-polystability of Q-Fano varieties admitting Kähler-Einstein metrics}

\author{Robert J. Berman}
\begin{abstract}
It is shown that any, possibly singular, Fano variety $X$ admitting
a Kähler-Einstein metric is K-polystable, thus confirming one direction
of the Yau-Tian-Donaldson conjecture in the setting of $\Q$-Fano
varieties equipped with their anti-canonical polarization. The proof
is based on a new formula expressing the Donaldson-Futaki invariants
in terms of the slope of the Ding functional along a geodesic ray
in the space of all bounded positively curved metrics on the anti-canonical
line bundle of $X.$ One consequence is that a toric Fano variety
$X$ is K-polystable iff it is K-polystable along toric degenerations
iff $0$ is the barycenter of the canonical weight polytope $P$ associated
to $X.$ The results also extend to the logarithmic setting and in
particular to the setting of Kähler-Einstein metrics with edge-cone
singularities. Applications to geodesic stability, bounds on the Ricci
potential and Perelman's $\lambda-$entropy functional on $K-$unstable
Fano manifolds are also given.
\end{abstract}

\address{Department of Mathematical Sciences, Chalmers University of Technology
and University of Gothenburg, Sweden}

\email{robertb@chalmers.se}

\maketitle
\tableofcontents{}

\section{Introduction}

Let $(X,L)$ be a polarized projective algebraic manifold. i.e. $L$
is an ample line bundle over $X.$ According to the fundamental Yau-Tian-Donaldson
conjecture in Kähler geometry (see the recent survey \cite{p-s2})
the first Chern class $c_{1}(L)$ contains a Kähler metric $\omega$
with \emph{constant scalar curvature} if and only if $(X,L)$ is \emph{K-polystable.}
This notion of stability is of an algebro-geometric nature and has
its origin in Geometric Invariant Theory (GIT). It was introduced
by Tian \cite{ti1} and in its most general form, due to Donaldson
\cite{d0} it is formulated in terms of polarized $\C^{*}-$equivariant
deformations $\mathcal{L}\rightarrow\mathcal{X}\rightarrow\C$ of
$(X,L)$ called\emph{ test configurations }for the polarized variety
$(X,L),$ where $\mathcal{X}_{1}=X.$\emph{ }Briefly, to any test
configuration $(\mathcal{X},\mathcal{L})$ one associates a numerical
invariant $DF(\mathcal{X},\mathcal{L}),$ called the \emph{Donaldson-Futaki
invariant }defined in terms of the polarized scheme $(\mathcal{X}_{0},\mathcal{L}_{|\mathcal{X}_{0}})$\emph{
}and $X$ is said to be K-polystable if $DF(\mathcal{X},\mathcal{L})\geq0$
with equality if and only if\emph{ $(\mathcal{X},\mathcal{L})$ }is
isomorphic to a product test configuration (the precise definitions
are recalled in section \ref{sub:K-polystability-and-test}). The
test configuration $(\mathcal{X},\mathcal{L})$ thus plays the role
of a one-parameter subgroup in GIT and the Donaldson-Futaki invariant
corresponds to the Hilbert-Mumford weight in GIT. Accordingly, the
Yau-Tian-Donaldson conjecture is sometimes also referred to as the
non-linear version of the celebrated Kobayashi-Hitchin correspondence
between Hermitian Yang-Mills metrics and polystable vector bundles. 

In the case when the connected component $\mbox{Aut}(X)_{0}$ containing
the identity of the the automorphism group is trivial, i.e. $X$ admits
no non-trivial holomorphic vector fields, it was shown by Stoppa \cite{st}
that the existence of a constant scalar curvature metric in $c_{1}(L)$
indeed implies that $(X,L)$ is K-polystable. The case when $\mbox{Aut}(X)_{0}$
is non-trivial leads to highly non-trivial complications, related
to the case when $DF=0$ and was treated by Mabuchi in a series of
papers \cite{ma2,ma3} (see also \cite{st-sz} where it is shown that
if $c_{1}(L)$ contains an extremal Kähler metric, then $(X,L)$ is
K-polystable with respect to all test configurations whose $\C^{*}-$action
commutes with a maximal torus of automorphisms). In this note we will
be concerned with the special case when $\omega$ is a \emph{Kähler-Einstein
metric} of positive scalar curvature. Equivalently this means that
the Ricci curvature of $\omega$ is positive and constant: 
\[
\mbox{Ric }\ensuremath{\omega=\omega,}
\]
i.e. $L$ is the anti-canonical line bundle $-K_{X}$ and $X$ is
a Fano manifold. In the seminal paper of Tian \cite{ti1} it was shown,
in the case when $\mbox{Aut}(X)_{0}$ is trivial, that $X$ is K-stable
along all test configurations $\mathcal{X}$ with normal central fiber
$\mathcal{X}_{0}.$ Here we will show that the assumption on $\mbox{Aut}(X)_{0}$
can be removed, as well as the normality assumption on the central
fiber $\mathcal{X}_{0}.$ In fact, we will allow $X$ to be a general,
possibly singular, Fano variety and prove the following
\begin{thm}
\label{thm:k-poly intro}Let $X$ be a Fano variety admitting a Kähler-Einstein
metric. Then $X$ is K-polystable.
\end{thm}
It should be pointed out that, following Li-Xu \cite{l-x}, we assume
that the total space $\mathcal{X}$ of the test configuration is normal
to exclude some pathological test-configurations that had previously
been overlooked in the literature (as explained in \cite{l-x} ).
As follows from the results of Ross-Thomas \cite{r-t} this does not
affect the notion of K semi-stability. Moreover, by a remark of Stoppa
\cite{st-1} K-polystability for all normal test configuration is
equivalent to having $DF(\mathcal{X},\mathcal{L})\geq0$ for all test
configurations with equality iff $(\mathcal{X},\mathcal{L})$ is isomorphic
to a product test configuration away from a subvariety of codimension
at least two.

We recall that, by definition, $X$ is a Fano variety if it is normal
and the anti-canonical divisor $-K_{X}$ is defined as an ample $\Q-$line
bundle (such a variety is also called a $\Q-$Fano variety in the
literature) and, following \cite{bbegz}, $\omega$ is said to be
a\emph{ Kähler-Einstein metric} on $X$ if $\omega$ is a bona fide
Kähler-Einstein metric on the regular locus $X_{reg}$ of $X$ and
the volume of $\omega$ on $X_{reg}$ coincides with the top-intersection
number $c_{1}(-K_{X})^{n}[X].$ The existence of such a metric actually
implies that the singularities are rather mild in the sense of the
Minimal Model Program (MMP) in birational geometry \cite{bbegz},
more precisely the singularities of $X$ are \emph{(Kawamata) log
terminal} \emph{(klt}, for short). In fact, even if $X$ itself is
smooth we will show that the singularity structure of the central
fiber $\mathcal{X}_{0}$ of a given test configuration for $X$ (or
more precisely the log canonical threshold of $\mathcal{X}_{0}$)
plays an important role in the metric analysis of the Donaldson-Futaki
invariant $DF(\mathcal{X},\mathcal{L}),$ through the Lelong number
$l_{0}$ at $0$ of the $L^{2}-$ type metric on a certain adjoint
direct image sheaf over the base of the test configuration. Interestingly,
this will single out test configurations $\mathcal{X}$ whose central
fibers have log canonical singularities as the ``minimal'' ones,
in the sense that $l_{0}=0.$ This appears to give a new metric incarnation
of the appearance of log canonical singularities in the log MMP (which
is reminiscent of the algebro-geometric results in \cite{l-x}; compare
Remark \ref{Rem:li-x}).

One motivation for considering singular Kähler-Einstein varieties
$X$ is that they naturally appear when taking Gromov-Hausdorff limits
of smooth Kähler-Einstein varieties \cite{do-s}. This is related
to the expectation that one may be able to form \emph{compact }moduli
spaces of K-polystable Fano varieties if singular ones are included,
or more precisely those with log terminal singularities; compare the
discussions in \cite{od-3} and \cite{od-s-s} (where the surface
case is considered). 

Another motivation for allowing $X$ to be singular comes from the
toric setting considered in \cite{b-b-2}, where it was shown that
the existence of a Kähler-Einstein metric on a toric Fano variety
$X$ is equivalent to $X$ being $K-$polystable with respect to\emph{
toric} test configuration. In turn, this latter property is equivalent
to the canonical rational weight polytope $P$ associated to $X$
having zero as its barycenter. However, the question whether the existence
of a Kähler-Einstein metric on the toric variety $X$ implies that
$X$ is K-polystable for \emph{general} test configurations was left
open in \cite{b-b-2}. Combining the previous theorem with the results
in \cite{b-b-2} we thus deduce the following
\begin{cor}
A toric Fano variety is K-polystable iff it is K-polystable with respect
to toric test configurations. In particular, if $P$ is a reflexive
lattice polytope, then the toric Fano variety $X_{P}$ associated
to $P$ is $K-$polystable if and only if $0$ is the barycenter of
$P.$ 
\end{cor}
We recall that\emph{ reflexive} lattice polytopes $P$ (i.e. those
for which the dual $P^{*}$ is also a lattice polytope) correspond
to toric Fano varieties whose singularities are Gorenstein, i.e $-K_{X}$
is an ample line bundle (and not only a $\Q-$line bundle). This huge
class of lattice polytopes plays an important role in string theory,
as they give rise to many examples of mirror symmetric Calabi-Yau
manifolds \cite{ba}. Already in dimension three there are 4319 isomorphism
classes of such polytopes \cite{k-s} and hence including \emph{singular}
Fano varieties leads to many new examples of K-polystable and K-unstable
Fano threefolds (recall that there are, all in all, only 105 families
of \emph{smooth }Fano threefolds). 

As explained in section \ref{sub:The-logarithmic-setting} the theorem
above extends to the logarithmic setting of Kähler-Einstein metrics
on \emph{log Fano varieties} $(X,D),$ as considered in \cite{bbegz}.
In particular, this shows that if $D$ is an effective $\Q-$divisor
with simple normal crossings, and coefficients $<1,$ on a projective
manifold $X$ such that the logarithmic first Chern class of $(X,D)$
contains a Kähler-Einstein metric $\omega$ with\emph{ edge-cone singularities}
along $D$ in the sense of \cite{do-3,cgh,j-m-r}, then the pair $(X,D)$
is log K-polystable in the sense of \cite{do-3,li1,o-s}.

The starting point of the proof of Theorem \ref{thm:k-poly intro}
is the following result of independent interest, which expresses the
Donaldson-Futaki invariant in terms of the Ding functional $\mathcal{D}$
(see formula \ref{eq:def of ding functional}):
\begin{thm}
\label{thm:DF=00003Dding intro}Let $X$ be a Fano variety with log
terminal singularities, \emph{$(\mathcal{X},\mathcal{L})$ a test
configuration for $(X,-K_{X})$ (assumed to have normal total space)
and $\phi$ a locally bounded metric on $\mathcal{L}$ with positive
curvature current. Setting $t:=-\log|\tau|^{2}$ and denoting by $\phi^{t}=\rho(\tau)^{*}\phi_{\tau}$
the corresponding ray of locally bounded metrics on $-K_{X}$ the
following formula holds: }
\begin{equation}
DF(\mathcal{X},\mathcal{L})=\lim_{t\rightarrow\infty}\frac{d}{dt}\mathcal{D}(\phi^{t})+q,\label{eq:df in terms of ding in theorem intro}
\end{equation}
 where $q$ is a non-negative rational number determined by the polarized
scheme $(\mathcal{X}_{0},\mathcal{L}_{|\mathcal{X}_{0}})$ with the
following properties, if $X$ is smooth:
\begin{itemize}
\item $q=0$ iff $\mathcal{X}$ is $\Q-$Gorenstein with $\mathcal{L}$
isomorphic to $-K_{\mathcal{X}/\C}$ and $\mathcal{X}_{0}$ is reduced
and its normalization has log canonical singularities.
\item In particular, if $\mathcal{X}_{0}$ is normal then $q=0$ iff $\mathcal{X}_{0}$
is reduced and has log canonical singularities (and in particular
$q=0$ if it has log terminal singularities).
\end{itemize}
\end{thm}
More precisely, we will give an explicit expression for the number
$q$ in terms of a given log resolution $(\mathcal{X},\mathcal{X}_{0})$
(see Theorem \ref{thm:df=00003Dding}). In order to prove Theorem
\ref{thm:k-poly intro} we apply Theorem \ref{thm:DF=00003Dding intro}
to a weak geodesic ray $\phi^{t}$, emanating from the Kähler-Einstein
metric on $-K_{X}$ (which is a critical point of the Ding functional).
We can then exploit a result of Berndtsson \cite{bern2} (and its
generalization to singular Fano varieties in \cite{bbegz}) concerning
the convexity of the Ding functional $\mathcal{D},$ also using a
new triviality result for certain flat direct image sheaves, of independent
interest (Proposition \ref{prop:flat vector bundle}). 

As for the proof of Theorem \ref{thm:DF=00003Dding intro} it is based
on the observation that $\mathcal{D}(\phi^{t})$ extends to define
a singular positively curved metric on a certain line bundle over
the base $\C$ of the given test configuration, that we will accordingly
call the \emph{Ding line bundle}. To make the connection to $DF(\mathcal{X},\mathcal{L})$
we use a result of Phong-Ross-Sturm \cite{p-r-s} which expresses
$DF(\mathcal{X},\mathcal{L})$ in terms of the weight over $0$ of
another line bundle $\eta$ over the base $\C,$ involving certain
Deligne pairings. This is also closely related to the intersection
theoretic formulation of the Donaldson-Futaki invariant due to Wang
\cite{w} and Odaka \cite{od}, independently. The error term $q$
in formula \ref{eq:df in terms of ding in theorem intro} can then
be decomposed into two pieces where the first piece is the Lelong
number $l_{0}$ at $0$ of Ding metric referred to above, which is
shown to coincide with the Lelong number at $0$ of an $L^{2}-$type
metric on a certain direct image sheaf. The non-negativity of $l_{0}$
then follows from the positivity results of Berndtsson-Paun for the
$L^{2}-$metrics on direct images of adjoint line bundles \cite{bern1,be-p}.
We show how to express $l_{0}$ explicitly in terms of a certain log
canonical threshold of the central fiber$\mathcal{X}_{0}$ (Proposition
\ref{prop:lelong of direct image}). Finally, the vanishing properties
of $q$ are obtained using inversion of adjunction \cite{kawi} (which
can be seen as an algebro-geometric incarnation of the Ohsawa-Takegoshi
extension theorem in complex analysis \cite{d-k,ko}). 

It should be pointed out that the information about the vanishing
properties of $q$ in Theorem \ref{thm:DF=00003Dding intro} are not
used in the proof of Theorem \ref{thm:k-poly intro}, but they appear
to give a new link between differential geometry and the MMP (see
Remark \ref{Rem:li-x}). Moreover, as discussed in section \ref{sec:Outlook-on-the},
the second point in Theorem \ref{thm:DF=00003Dding intro} fits naturally
into Tian's program \cite{ti2} for establishing the existence part
of the Yau-Tian-Donaldson conjecture - in particular when generalized
to the setting of singular Fano varieties (compare Corollary \ref{cor:mab diverges}).

We also give some applications of Theorem \ref{thm:DF=00003Dding intro}
to bounds on the Ricci potential and Perelman's $\lambda-$entropy
functional \cite{pe} (see section \ref{sub:Applications-to-bounds}),
which can be seen as analogs of Donaldson's lower bound on the Calabi
functional \cite{do2}. In particular, we obtain the following 
\begin{thm}
\label{thm:perelman intro}Let $X$ be an $n-$dimensional Fano manifold
and set $V:=c_{1}(X)^{n}.$ If $X$ is $K-$unstable, then Perelman's
$\lambda-$entropy functional satisfies 
\[
\sup_{\omega\in\mathcal{K}(X)}\lambda(\omega)<nV,
\]
 where $\mathcal{K}(X)$ denotes the space of all Kähler metrics in
$c_{1}(X).$
\end{thm}
As is well-known $\lambda(\omega)\leq nV$ on the space $\mathcal{K}(X)$
and, as recently shown by Tian-Zhang \cite{ti-zhu} in their study
of the Kähler-Ricci flow, if a Fano manifold $X$ admits a Kähler-Einstein
metric $\omega_{KE}$ then $\lambda(\omega_{KE})=nV,$ or more generally:
if Mabuchi's K-energy is bounded from below on $\mathcal{K}(X),$
then supremum of $\lambda$ is equal to $nV.$ In the light of the
Yau-Tian-Donaldson conjecture it seems thus natural to conjecture
that $X$ is $K-$semistable if and only if the supremum of $\lambda$
is equal to $nV$ (the ``if direction'' is the content of the previous
theorem). In fact, a more precise version of Theorem \ref{thm:perelman intro}
will be obtained, where the supremum of $\lambda$ is explicitly bounded
in terms of minus the supremum of the Donaldson-Futaki invariants
over all (normalized) destabilizing test configurations for $(X,L)$
(see Cor \ref{cor:bound on lambda f}). 

During the revision of the present paper the existence of Kähler-Einstein
metrics on K-polystable smooth Fano varieties was finally settled
by Chen-Donaldson-Sun \cite{c-d-s} based on a modification of Tian's
original program introduced by Donaldson \cite{do-3}, which uses
metrics with conical singularities. In fact the proofs in \cite{c-d-s}
show that a Kähler-Einstein metric exists as soon as $X$ is K-polystable
with respect to \emph{special} test configurations and hence combining
the results in \cite{c-d-s} with the main result of the present paper
yields a new proof - not involving MMP - of the recent result of Li-Xu
\cite{l-x}, saying that to test the $K-$polystability of a Fano
manifolds it is enough to test it on special test configurations.
Moreover, one also obtains a proof of an analog of a conjecture of
Donaldson concerning ``geodesic stability'' saying that either a
Fano manifold $X$ admits a Kähler-Einstein metrics, or there exists
a geodesic ray along which the Ding functional eventually becomes
strictly decreasing (Theorem \ref{thm:don conj for ding}).

In section \ref{sec:Outlook-on-the} we have included an outlook on
the existence problem on singular Fano varieties, which is thus the
missing piece in the Yau-Tian-Donaldson conjecture for projective
varieties $X$ polarized by $-K_{X}.$ The opposite case of varieties
polarized by $K_{X}$ was very recently established in \cite{b-g},
building on \cite{od}.

\subsection{\label{sub:Further-relations-to}Further relations to previous work}

In the case when $X$ is a smooth Fano manifold Theorem \ref{thm:DF=00003Dding intro}
(and its more precise version Theorem \ref{thm:df=00003Dding}) should
be compared with previous results of Ding-Tian \cite{di-ti} who considered
the case when $\phi^{t}$ is a Bergman geodesic, induced by a fixed
embedding in $\P^{N}$ by $-kK_{X}$ (and a $\C^{*}-$action on $\P^{N}).$
In the case when the central fiber $\mathcal{X}_{0}$ is normal the
results of Ding-Tian say that $DF(\mathcal{X},\mathcal{L})$ is equal
to the asymptotic slope of the Mabuchi functional (without any further
restrictions on the nature of the singularities of $\mathcal{X}_{0}).$
We also recall that in another direction Paul-Tian \cite[I, Cor 1.2]{p-t}
and Phong-Ross-Sturm \cite{p-r-s} considered the case of a general
smooth and positively curved metric $\phi$ on $\mathcal{L}\rightarrow\mathcal{X},$
for a given test configuration $(\mathcal{X},\mathcal{L})$ for a
polarized manifold $(X,L),$ but assumed that the total space $\mathcal{X}$
be smooth and then obtained a formula for $DF(\mathcal{X},\mathcal{L})$
as the slope of the Mabuchi functional plus a correction term which
vanishes if $\mathcal{X}_{0}$ is reduced. 

It may also be illuminating to compare our proof of Theorem \ref{thm:k-poly intro}
with the original approach of Tian \cite{ti1} in the case of a non-singular
Fano variety. As shown by Tian \cite{ti1} assuming that the central
fiber of test configuration is normal, and using the slope formula
of Ding-Tian \cite{di-ti} referred to above, the Donaldson-Futaki
invariant $DF$ is expressed in terms of the asymptotics of \emph{Mabuchi's
K-energy }functional along a one-parameter family $\phi_{k}^{t}$
of Bergman metrics, i.e. a Bergman geodesic. The positivity properties
of $DF$ are then determined using that, in the presence of a Kähler-Einstein
metric, the Mabuchi's K-energy functional is \emph{proper} (if there
are no non-trivial holomorphic vector fields on $X),$ which is the
content of deep result of Tian \cite{ti1}. Here we thus show that
the Mabuchi functional and the Bergman geodesic may be replaced by
the Ding functional and a weak (bounded) geodesic, respectively, and
the properness result with Berndtsson's convexity result. One technical
advantage of the Ding functional is that, unlike the Mabuchi functional,
it is indeed well-defined along a weak geodesic, as previously exploited
in \cite{bern2,bbegz} in the context of the uniqueness problem for
Kähler-Einstein metrics. Thus the approach in this paper is in line
with the programs of Phong-Sturm \cite{p-s} and Chen-Tang \cite{ch}
for calculating Donaldson-Futaki invariants by using (weak) geodesic
rays associated to test configurations.

In the case when $X$ is a smooth Kähler-Einstein Fano variety with
$\mbox{Aut}(X)_{0}$ trivial the properness of the Ding functional
was shown by Tian \cite{ti1} as a consequence of his properness result
for the Mabuchi functional. It was later shown in \cite{p-s-s-w}
that if center of the group $\mbox{Aut}(X)_{0}$ is finite then the
Ding functional is still proper (in an appropriate sense), but the
properness in the case of general Kähler-Einstein manifold is still
open. The generalization of the properness result (even when $\mbox{Aut}(X)_{0}$
is trivial) to singular Fano varieties and more generally log Fano
varieties also appears to be a challenging open problem. Anyway, these
subtle issues are bypassed in the present approach.

\subsubsection*{Organization }

After having recalled some preliminary material in Section \ref{sec:Preliminaries}
the formula relating the Donaldson-Futaki invariant (Theorem \ref{thm:DF=00003Dding intro}
above) to the Ding function is established in Section \ref{sec:Singularity-structure-of}
and then used, by exploiting positivity results for $L^{2}-$type
metrics on direct images, to prove Theorem \ref{thm:k-poly intro}
concerning K-polystability. Section \ref{sub:Singularity-structure-of}
also contains a detailed study of the singularities of the $L^{2}-$type
metrics which is of independent interest, but not needed for the proof
of Theorem \ref{thm:k-poly intro}. In section \ref{sub:Applications-to-bounds}
various ramifications and applications are given to (i) an analog
of conjecture of Donaldson (ii) bounds on the Ricci potential and
Perelman's entropy functional and (iii) the log Fano setting. The
paper is concluded with an outlook in Section \ref{sec:Outlook-on-the}
on the existence problem for Kähler-Einstein metrics on singular Fano
varieties is given.

\subsubsection*{Acknowledgments}

Thanks to Bo Berndtsson, Sébastien Boucksom, Dennis Eriksson, Yuji
Odaka, Julius Ross and Song Sun for helpful discussions and comments.
In particular, thanks to Tomoyuki Hisamoto and David Witt-Nyström
for discussions on norms of test configurations and the relations
to their works \cite{hi} and \cite{n} (compare Remark \ref{rem:Lemma--infty norm}).
Finally, thanks to the referees for many useful comments. This work
was supported by grants from the Swedish Research Council and the
European Research Council.

\section{\label{sec:Preliminaries}Preliminaries }

In this section we will setup the notation and recall the basic tools
to be used in the proofs of the main results.

\subsection{Kähler-Einstein metrics on Fano varieties and log pairs}

\subsubsection{\label{sub:Fano-varieites-and}Fano varieties and log pairs}

Let $X$ be an $n-$dimensional normal compact projective algbraic
variety. In analytic terms normality just means that any holomorphic
function $f$ on $U\cap X_{reg},$ where $X_{reg}$ denotes the regular
locus of $X$ and $U\subset X$ is open, extends holomorphically to
$U$ in the strong sense, i.e. $f$ is the restriction to $U$ of
holomorphic function in $\C^{m}$ under some local holomorphic embedding
$F:\, U\hookrightarrow\C^{m}$ (after perhaps shrinking $U).$ In
particular, $\mbox{codim}(X-X_{reg})\geq2.$ More generally, the corresponding
strong extension property holds for any plurisubharmonic (psh, for
short) function $\phi$ on $U$ (see \cite{b-g} and references therein
for a more detailed discussion of pluripotential theory on singular
analytic spaces). By definition, $X$ is said to be a\emph{ Fano variety}
if the anti-canonical line bundle $-K_{X}:=\det(TX)$ defined on the
regular locus $X_{reg}$ of $X$ extends to an ample $\Q-$line bundle
on $X,$ i.e. there exists a positive integer $m$ such that the $m$th
tensor power $-mK_{X_{reg}}$ extends to an ample line bundle over
$X.$ Since $X$ is normal this equivalently means that the anti-canonical
divisor $-K_{X}$ of $X$ defines an ample $\Q-$line bundle. More
generally we will (in particular in Section \ref{sub:The-logarithmic-setting})
consider \emph{log pairs }$(X,D)$ in the sense of birational geometry
\cite{ko}: i.e. $X$ is normal and $D$ is a $\Q-$divisor on $X$
such that $K_{X}+D$ is a $\Q-$Cartier, i.e. defines a $\Q-$line
bundle (called the log canonical bundle of $(X,D)$. By definition,
a \emph{log resolution} of a log pair $(X,D)$ is a proper birational
morphism $X'\rightarrow X$ such that $p^{*}D+E$ has simple normal
crossings, where $E$ is the exceptional divisor of $p.$ Then 
\begin{equation}
p^{*}(K+D)=K_{X'}+D',\label{eq:pull-back of log can}
\end{equation}
 for a $\Q-$divisor $D'$ on $X'$ (by Hironaka's theorem we may
and will assume that $p$ is an isomorphism away from $p^{-1}(X_{sing}\cup\mbox{Supp}D_{sing}).$
A log pair $(X,D)$ is said to be \emph{log canonical, or lc }for
short\emph{,} if the coefficients $c_{i}$ of $D'$ (along the corresponding
prime divisors) satisfy $c_{i}\leq1.$ Similarly, $(X,D)$ is said
to be \emph{(Kawamata) log terminal, }or\emph{ klt for short, if $c_{i}<1.$
}Setting $D=0$ these notations also apply to the normal variety $X,$
which is thus said to have \emph{log canonical (log terminal) singularities}
if $(X,0)$ is log canonical (log terminal). In practice we will in
what follows only consider Fano varieties $X$ with log terminal singularities
(the corresponding analytical characterization will be recalled below),
but even if $X$ is smooth the notion of log canonical singularities
will be important in the study of test configurations $\mathcal{X}$
for $X.$

\subsubsection{Singular metrics on line bundles and (multi-) sections}

Throughout the paper we will use additive notation for line bundles,
as well as metrics. This means that a metric $\left\Vert \cdot\right\Vert $
on a line bundle $L\rightarrow X$ is represented by a collection
of local functions $\phi(:=\{\phi_{U}\})$ defined as follows: given
a locally trivializing section of $L,$ i.e. a local generator $s_{U}$
of the invertible sheaf $\mathcal{O}(L)$ on an open subset $U\subset X$
we set $\phi_{U}:=-\log\left\Vert s_{U}\right\Vert ^{2},$ where $\phi_{U}$
is upper semi-continuous. It will be convenient to identify the additive
object $\phi$ with the metric it represents. Of course, $\phi_{U}$
depends on $s_{U}$ but the curvature current 
\[
dd^{c}\phi:=\frac{i}{2\pi}\partial\bar{\partial}\phi_{U}
\]
 is globally well-defined on $X$ and represents the first Chern class
$c_{1}(L),$ which with our normalizations lies in the integer lattice
of $H^{2}(X,\R).$ We will say that a (singular) metric $\phi$ is
\emph{psh }(or have\emph{ positive curvature current}), $\phi\in PSH(X,L),$
if $\phi_{U}$ is always psh (and in particular $dd^{c}\phi\geq0$
holds in the sense of currents). By the normality of $X$ the injection
$X_{reg}\hookrightarrow X$ alows us to identity $PSH(X,L)=PSH(X_{reg},L)$
and hence, given a smooth resolution $\pi:\, X'\rightarrow X,$ we
can also identify $PSH(X,L)$ with $PSH(X',\pi^{*}L),$ using the
pull-back $\pi^{*}.$ We will denote by $\mathcal{H}_{b}(X,L)$ the
subspace of $PSH(X,L)$ consisting of all \emph{locally bounded} metrics.
Fixing $\phi_{0}\in\mathcal{H}_{b}(X,L)$ and setting $\omega_{0}:=dd^{c}\phi_{0}$
the map $\phi\mapsto v:=\phi-\phi_{0}$ thus gives an isomorphism
between the space $\mathcal{H}_{b}(X,L)$ and the space $PSH(X,\omega_{0})\cap L^{\infty}(X)$
of all bounded $\omega_{0}-$psh functions, i.e. the space of all
bounded usc functions $v$ on $X$ such that $dd^{c}v+\omega_{0}\geq0.$
Similarly, a metric $\phi$ will be said to be\emph{ smooth }if $\phi_{U}$
is the restriction to $U$ of a smooth function under a local embedding
as above. A special class of smooth metrics with strictly positive
curvature is given by \emph{Bergman metrics}, i.e. metrics of the
form $\phi_{k}/k,$ where $\phi_{k}$ is obtained by restricting the
Fubini-Study metric $\phi_{FS}$ on $\P^{N}$ under a given Kodaira
embedding of $X$ in $\P^{N}$ induced by $kL,$ for some $k$ sufficiently
large.

In particular, if $s$ is a holomorphic section of $L\rightarrow X$
then $\phi_{s}:=\log|s|^{2}$ defines a singular metric on $L$ with
positive curvature current $[D],$ i.e. integration along the zero
divisor of $s$ (taking multiplicities into account). More generally,
it will often be convenient to use the terminiology of holomorphic
\emph{multisections} of $L,$ which by definition consists of a pair
$(r,s_{r}),$ where $r$ is a positive integer $r$ and $s_{r}\in H^{0}(X,rL)$
and where two pairs $(r,s_{r})$ and $(r',s_{r'}),$ where $r'\geq r,$
are identified if there exists a positive integer $p$ such that $r'=pr$
and $s_{r'}=s_{r}^{\otimes p}.$ Denoting by $s$ such an equivalence
class $\phi_{s}:=\frac{1}{r}\log|s_{r}|^{2}$ defines a singular metric
on $L$ with curvature current $[D],$ where $D$ is the $\Q-$divisor
defined by the zero-divisor of $s$ (i.e. $D$ is, by definition,
equal to $1/r$ times the zero divisor of $s_{r}).$ Accordingly we
will occasionally also write $\phi_{D}:=\phi_{s}.$ More generally,
abusing notation slightly, the statement ``let $s$ be a holomorphic
multisection of $L$'' will in the follopwing mean that we tacitly
fix a pair $(r,s_{r})$ defining $s$ and work with the bona fide
section $s_{r}$ and then make the appropriate scalings by $r.$

\subsubsection{\label{sub:canonical measures}Canonical measures }

In the special case when $L=-K_{X}$ any given metric on $\phi\in\mathcal{H}_{b}(X,L)$
induces a measure $\mu_{\phi}$ on $X,$ which may be concretely defined
as follows: if $U$ is a coordinate chart in $X_{reg}$ with local
holomorphic coordinates $z_{1},...,z_{n}$ we let $\phi_{U}$ be the
representation of $\phi$ with respect to the local trivialization
of $-K_{X}$ which is dual to $dz:=dz_{1}\wedge\cdots\wedge dz_{n}.$
Then we define the restriction of $\mu_{\phi}$ to $U\subset X_{reg}$
as
\[
\mu_{\phi}=e^{-\phi_{U}}i^{n^{2}}dz\wedge d\bar{z}
\]
 This expression is readily verified to be independent of the local
coordinates $z$ and hence defines a measure $\mu_{\phi}$ on $X_{reg}$
which we then extend by zero to all of $X.$ Note that since $-K_{X}$
is assumed $\Q-$Cartier we may cover $X$ with a finite number of
open sets $V$ (not necessarily contained in $X_{reg})$ such that
the restriction to $V$ of $\mu_{\phi}$ is given by $1_{X_{reg}}i^{n^{2}}\alpha_{U}\wedge\overline{\alpha_{U}}e^{-\phi_{U}},$
where $\alpha_{U}$ is a trivializing section of $K_{X|U}$ (whose
restriction to $U\cap X_{reg}$ may thus be identified with a holomorphic
$(n,0)-$form) and where $\phi_{U}=-\log\left\Vert s\right\Vert ^{2}$
for $s$ the dual of $\alpha.$ The Fano variety $X$ has\emph{ }log
terminal singularities (as defined above) precisely when the total
mass of $\mu_{\phi}$ is finite for some and hence any $\phi\in\mathcal{H}_{b}(X,L)$
(see Lemma \ref{lem:singular complex sing is lct}). Abusing notation
slightly we will often use the suggestive notation $e^{-\phi}$ for
the measure $\mu_{\phi}.$ This notation is compatible with the usual
notation used in the context of adjoint bundles: if $s$ is a holomorphic
section of $L+K_{X}\rightarrow X$ and $\phi$ is a metric on $L$
then $|s|^{2}e^{-\phi}$ (sometimes written as $i^{n^{2}}s\wedge\bar{s}e^{-\phi})$
may be naturally identified with a measure on $X.$ In particular,
letting $L=-K_{X}$ and taking $s$ to be the canonical section $1$
in the trivial line bundle $L+K_{X}$ gives us back the measure $\mu_{\phi}.$
More generally, if $(X,D)$ is a log pair (see section \ref{sub:The-logarithmic-setting}
below) and $\phi$ is a locally bounded metric on $-(K_{X}+D)$ then
one obtains a measure $\mu_{(X,D,\phi)}$ on $X$ by using the natural
identification between $-(K_{X}+D)$ and $-K_{X}$ on the complement
of the support of $D$ in $X$ and extending by zero to all of $X$
(compare \cite{bbegz}). Abusing notation, we will sometimes write
$\mu_{(X,D;\phi)}=e^{-(\phi+\log|s_{D}|^{2})},$ where $s_{D}$ is
the (multi-) section cutting out $D.$ These constructions are compatible
with taking resolutions $p,$ as in \ref{eq:pull-back of log can}:
if $\phi$ is a metric on $-(K_{X}+D)$ then $p^{*}\phi$ is a metric
on $-(K_{X'}+D')$ and $p_{*}(\mu_{(X',D';p^{*}\phi)})=\mu_{(X,D;\phi)}.$

In the relative setting of a morphism $\pi:\mathcal{X}\rightarrow\C$
from a normal $\Q-$Gorenstein variety such that $\pi$ is smooth
over $\C^{*}$ (or more generally, $\mathcal{X}_{\tau}$ is reduced
and defines a variety $X_{\tau}$ with log terminal singularities
for $\tau\neq0)$ we denote by $K_{\mathcal{X}/\C}:=K_{\mathcal{X}}-\pi^{*}K_{\C}$
the relative canonical line bundle (viewed as a $\Q-$line bundle).
Denoting by $\tau$ the standard affine coordinate on $\C$ we will
use $\pi^{*}d\tau$ to trivialize $\pi^{*}K_{\C}$ over $\mathcal{X}.$
Accordingly, we will identify an element $s\in\mathcal{O}(U,K_{\mathcal{X}/\C})$
with a holomorphic $(n+1)-$form $\alpha$ on $U\cap\mathcal{X}_{reg}$
(i.e. $s=\alpha\otimes\pi^{*}\frac{\partial}{\partial\tau}).$ Moreover,
if $U=\pi^{-1}(V)$ for $V\subset\C^{*},$ then the natural isomorphism
$K_{\mathcal{X}/\C|X_{\tau}}\approx K_{X_{\tau}}$ allows us to identify
the restrictions $s_{\tau}$ of $s\in\mathcal{O}(U,K_{\mathcal{X}/\C})$
to $X_{\tau}$ with a family of holomorphic $n-$forms on $X_{\tau}$
(i.e. $\alpha=s_{\tau}\wedge\pi^{*}d\tau).$ Similarly, replacing
$K_{\mathcal{X}/\C}$ with $\mathcal{L}+K_{\mathcal{X}/\C}$ for a
given line bundle $\mathcal{L}\rightarrow\mathcal{X}$ equipped with
a metric $\phi$ we will use the symbolic notation $|s|^{2}e^{-\phi}$
for the corresponding measures on $U\subset\mathcal{X}$ and $|s_{\tau}|^{2}e^{-\phi_{\tau}}$
for the corresponding family of measures over $V\subset\C^{*}.$

\subsubsection{Kähler-Einstein metrics}

Following \cite{bbegz} $\omega$ is said to be a \emph{Kähler-Einstein
metric} on $X$ if it is Kähler metric on $X_{reg}$ with constant
Ricci curvature, i.e. $\mbox{Ric \ensuremath{\omega=\omega}}$ on
$X_{reg}$ and $\int_{X_{reg}}\omega^{n}=c_{1}(-K_{X})^{n}.$ By  \cite[Lemma 3.6 and  Proposition 3.8]{bbegz}
this equivalently means that the Fano variety $X$ in fact has log
terminal singularities and $\omega$ extends to a Kähler current defined
on the whole Fano variety $X,$ such that $\omega$ is the curvature
current of a locally bounded (and in fact continuous) metric $\phi_{KE}$
on the $\Q-$line bundle $-K_{X}$ such that 
\begin{equation}
(dd^{c}\phi_{KE})^{n}=Ve^{-\phi_{KE}}/\int_{X}e^{-\phi_{KE}}.\label{eq:k-e equatio for phi in def}
\end{equation}
The measure appearing the left hand side above is the\emph{ Monge-Ampère
measure }of $\phi_{KE}$ defined in sense of pluripotential theory,
i.e. using the Bedford-Taylor product between positive closed currents
with locally bounded potentials (see \cite{bbegz} and references
therein for the general singular setting).

\subsection{\label{sub:K-polystability-and-test}K-polystability and test configurations}

Let us start by recalling Donaldson's general definition \cite{d0}
of K-stability of a polarized variety $(X,L)$ generalizing the original
definition of Tian \cite{ti1}. First, a general\emph{ test configuration}
$(\mathcal{X},\mathcal{L},\pi,\rho)$ for $(X,L)$ consists of a scheme
$\mathcal{X}$ with a $\C^{*}-$equivariant flat surjective morphism
$\pi:\,\mathcal{X}\rightarrow\C$ (where the base $\C$ is equipped
with its standard $\C^{*}-$action) and a relatively ample line bundle
$\mathcal{L}\rightarrow\mathcal{X}$ with a $\C^{*}-$action $\rho$
on $\mathcal{L}$ and such that $(X_{1},\mathcal{L}_{|_{X_{1}}})=(X,rL)$
for some integer $r.$ In fact, by allowing $\mathcal{L}$ to be a
$\Q-$line bundle we may as well assume that $r=1.$ More specifically,
following \cite{l-x} we will assume that the total space $\mathcal{X}$
is normal. Then the morphism $\pi$ is automatically flat \cite[Prop 9.7 ]{ha}). 

To simplify the notation we will usually surpress the dependence on
$\pi$ and $\rho$ denote a test configuration by $(\mathcal{X},\mathcal{L}).$
Occasionally we will use the notation $X_{0}$ for the reduction of
the central fiber $\mathcal{X}_{0},$ i.e. the projective variety
$X_{0}$ underlying the scheme theoretic central fiber $\mathcal{X}_{0}.$
For a\emph{ semi-test configuration} we only require that $\mathcal{L}$
be relatively semi-ample. We recall that the total space $\mathcal{X}$
of a test configuration may, using the relative linear systems defined
by $r\mathcal{L}$ for $r$ sufficiently large, be equivariantly embedded
as a subvariety of $\P^{N}\times\C$ so that $r\mathcal{L}$ becomes
the pull-back of the relative $\mathcal{O}(1)-$hyperplane line bundle
over $\P^{N}\times\C.$ We will denote by $\phi_{FS}$ the metric
on $\mathcal{L}$ obtained by restriction of the fiberwise Fubini-Study
metrics on $\P^{N}\times\{\tau\}$ (see \cite[Proposition 3.7]{r-t}
and the beginning of Section 5 \cite{do2} ).

The \emph{Donaldson-Futaki invariant} $DF(\mathcal{X},\mathcal{L})$
of a test configuration is defined as follows: consider the $N_{k}-$dimensional
space $H^{0}(\mathcal{X}_{0},k\mathcal{L}_{|\mathcal{X}_{0}})$ over
the central fiber $\mathcal{X}_{0}$ and let $w_{k}$ be the weight
of the $\C^{*}-$action on the complex line $H^{0}(\mathcal{X}_{0},k\mathcal{L}_{|\mathcal{X}_{0}}).$
Then the Donaldson-Futaki invariant of $DF(\mathcal{X},\mathcal{L})$
is defined as minus the sub-leading coefficient in the expansion of
$w_{k}/kN_{k}$ in powers of $1/k$ (up to normalization): 
\[
\frac{w_{k}(\det H^{0}(\mathcal{X}_{0},k\mathcal{L}_{|\mathcal{X}_{0}}))}{kN_{k}}=c_{0}-\frac{1}{k}\frac{1}{2}DF(\mathcal{X},\mathcal{L})+O(\frac{1}{k^{2}}),
\]
 where $N_{k}:=\dim(H^{0}(X_{0},kL_{0}).$ The polarized variety $(X,L)$
is said to be\emph{ K-semistable} if, for any test configuration,
$DF(\mathcal{X},\mathcal{L})\geq0$ and \emph{K-(poly)stable} if moreover
equality holds iff $(\mathcal{X},\mathcal{L})$ is (equivariantly)
isomorphic to $(X\times\C,p_{1}^{*}L)$. We also recall that $(X,L)$
is said to be\emph{ $K-$unstable }if it is not K-semistable, i.e.
there exists a \emph{destabilizing test configuration} in the sense
that $DF(\mathcal{X},\mathcal{L})<0.$
\begin{example}
\label{ex:product test}Let $V$ be a holomorphic vector field on
$X$ of type $(1,0)$ with a fixed lift to $L\rightarrow X.$ We will
say that $V$ \emph{generates a $\C^{*}-$action on $X,$ }denoted
by $\rho_{(X,V)}$, if $\frac{d}{dt}\rho_{(X,V)}(e^{-t/2})=e^{t\mbox{Re}V}$,
for $t\in\R$ (in other words the standard additive group $(\C,+)$
action on $X$ determined by the complex flow of $V$ descends to
a multiplicative action of $\C^{*}$ on $(X,L)$ under the homomorphism
$\C\rightarrow\C^{*},t\mapsto e^{-t/2}).$ Such a vector field $V$
determines a product test configuration $(\mathcal{X},\mathcal{L},\pi,\rho)$
by setting $(\mathcal{X},\mathcal{L},\pi)=(X\times\C,\mathcal{L=}p_{1}^{*}L,p_{2})$
and defining the action $\rho:\C^{*}\times\mathcal{X}\rightarrow\mathcal{X}$
by $(\lambda,(x,\tau))\mapsto(\rho_{V}(x),\lambda\tau).$ Note that
the original action of $\rho_{(X.V)}$ on $X$ may be identified with
the restricted action of $\rho$ on $\mathcal{X}_{0}$ and $DF(\mathcal{X},\mathcal{L},\rho)$
coincides with Futaki's invariant $F(X,V).$ Since, $F(X,V)=-F(X,V)$
a necessary condition for the K-polystability of $(X,L)$ is that
$DF(\mathcal{X},\mathcal{L},\rho)=0$ for any $V$ as above and a
necassary condition for K-stability is that $X$ admits no holomorphic
vector fields as above (since, $(\mathcal{X},\mathcal{L},\rho)$ is
equivariantly isomorphic to $(\mathcal{X},\mathcal{L},\rho_{triv}),$
where $\rho_{triv}(\lambda,(x,\tau)=(x,\lambda\tau),$ iff $V=0).$ 
\end{example}
In this paper we will be concerned with test configurations $(\mathcal{X},\mathcal{L})$
for a Fano variety with its anti-canonical polarization, i.e. $X$
is a Fano variety and $L=-K_{X}$ so that the restriction of $\mathcal{L}$
to the complement $\mathcal{X}^{*}$ of the central fiber coincides
with the $\Q-$line bundle defined by the dual of the relative canonical
divisor $K_{\mathcal{X}^{*}/\C}:=K_{\mathcal{X}^{*}}-\pi^{*}K_{\C}.$

\subsubsection{\label{sub:Special-test-configurations}$\Q-$Gorenstein and special
test configurations}

In general, $K_{\mathcal{X}/\C}$ does not extend as a $\Q-$line
bundle over the central fiber, but following \cite{l-x} we say that
a normal variety $\mathcal{X}$ with a $\C^{*}-$equivariant surjective
morphism $\pi$ to $\C$ is a\emph{ special test configuration for
the Fano variety $X$} if $\mathcal{X}_{1}=X,$ the total space $\mathcal{X}$
is $\Q-$Gorenstein and the central fiber is reduced and irreducible
and defines a Fano variety $X_{0}$ with log terminal singularities.
\begin{lem}
\label{lem:char of special test}Let $(\mathcal{X},\mathcal{L})$
be a general test configuration (with a priori non-normal total space)
for $(X,-K_{X}),$ where $X$ is a Fano variety. Assume that the central
fiber $\mathcal{X}_{0}$ is normal. Then $\mathcal{X}$ and $\mathcal{X}_{0}$
are both normal $\Q-$Gorenstein varieties and $\mathcal{L}_{|\mathcal{X}_{0}}$
is isomorphic to $-K_{\mathcal{X}_{0}},$ i.e. $\mathcal{L}$ is isomorphic
to $-K_{\mathcal{X}/\C}.$ Moreover, if $X_{0}$ has log terminal
singularities, then so has $\mathcal{X}.$ In other words, a test
configuration is special iff the central fiber is reduced and the
underlying variety $X_{0}$ has log terminal singularities.\end{lem}
\begin{proof}
This is essentially well-known, but for completeness we provide a
proof (thanks to Yuji Odaka for his help in this matter). It follows
from general commutative algebra that if $\pi:\,\mathcal{X}\rightarrow\C$
is a morphism projective and flat over $\C,$ with normal fibers,
then $\mathcal{X}$ is also normal \cite[Theorem 1.101 ]{g et al}.
In particular, the canonical divisor $K_{\mathcal{X}}$ is a well-defined
Weil divisor. By assumption $-K_{\mathcal{X}}$ and $\mathcal{L}$
are $\Q-$Cartier and linearly equivalent on $\mathcal{X}^{*}$ and
hence $K_{\mathcal{X}}+\mathcal{L}$ is linearly equivalent to a Weil
$\Q-$divisor $D$ supported in the central fiber. But the central
fiber is Cartier (since it is cut out by the function $\pi^{*}\tau)$
and hence, since it is assumed irreducible $-m(K_{\mathcal{X}}+\mathcal{L})$
is linearly equivalent to a multiple of $\mathcal{X}_{0},$ which
means that $-mK_{\mathcal{X}}$ is a sum of Cartier divisors, hence
Cartier, i.e. $\mathcal{X}$ is $\Q-$Gorenstein. More precisely,
$-mK_{\mathcal{X}}$ is linearly equivalent to $\mathcal{L}$ modulo
a pull back from the base and thus it follows from adjunction that
the restriction of $\mathcal{L}$ to $\mathcal{X}_{0}$ is linearly
equivalent to $-mK_{\mathcal{X}_{0}},$ which concludes the proof
of the first statement. Finally, if $\mathcal{X}_{0}$ has log terminal
singularities it  follows from inversion of adjunction that $\mathcal{X}$
also has log terminal singularities \cite[Theorem 7.5]{ko} (see also
the beginning of Section \ref{sec:Singularity-structure-of}). 
\end{proof}
Since the Donaldson-Futaki is independent of the lift of the $\C^{*}-$action
on $\mathcal{X}$ we may and will in the case when $\mathcal{L}:=-K_{\mathcal{X}/\C}$
assume that the $\C^{*}-$action on $\mathcal{L}:=-K_{\mathcal{X}/\C}$
is the canonical lift of the $\C^{*}-$action on $\mathcal{X}$ to
$-K_{\mathcal{X}/\C}.$

\subsection{\label{sub:Deligne-pairings-and}Deligne pairings, the energy functional
$\mathcal{E}$ and the line bundle $\eta$}

The Donaldson-Futaki invariant may also be expressed in terms of Deligne
pairings \cite{p-r-s} (also called intersection bundles \cite{m}).
First recall that if $\pi:\,\mathcal{X}\rightarrow B$ is a proper
flat projective morphism of relative dimension $n$ (between normal
schemes) and $L_{0},...,L_{n}$ are line bundles over $\mathcal{X}$
then the Deligne pairing $\left\langle L_{0},...,L_{n}\right\rangle $
is a line bundle over $B,$ which depends in a multilinear fashion
on $L_{i}$ \cite{zh,na} and satisfies 
\[
c_{1}\left\langle L_{0},...,L_{n}\right\rangle =\pi_{*}\left(c_{1}(L_{0})\wedge\cdots\wedge c_{1}(L_{n})\right)
\]
In particular, if $B$ is a non-singular projective curve then 
\begin{equation}
\deg\left\langle L_{0},...,L_{n}\right\rangle =L_{0}\cdots L_{n}\label{eq:deg of deligne is inters}
\end{equation}
In our case $\mathcal{X}$ will be a normal variety (defined over
$\C)$ and $B=\C.$ Given Hermitian metrics $\phi_{0},...,\phi_{n}$
on $L_{0},...,L_{n}$ there is a natural Hermitian metric $\left\langle \phi_{0},...,\phi_{n}\right\rangle $
on $\left\langle L_{0},...,L_{n}\right\rangle $ \cite{zh} which
has the following fundamental properties%
\footnote{Following \cite{e} the construction in \cite{zh} seems to require
that $\pi$ be Cohen-Macaulay in order to define the metric on $\left\langle L_{0},...,L_{n}\right\rangle $
by induction over the dimension $n.$ Anyway all our arguments will
be carried out on a non-singular resolution of $\mathcal{X}$ where
the constructions in \cite{zh,e} apply. %
}: 
\begin{enumerate}
\item Its curvature is given by 
\begin{equation}
dd^{c}\left\langle \phi_{0},...,\phi_{n}\right\rangle =\pi_{*}(dd^{c}\phi_{1}\wedge\cdots\wedge dd^{c}\phi_{n})\label{eq:curvature of the deligne pairing}
\end{equation}

\item If $\phi$ and $\psi$ are metrics in $\mathcal{H}(L)$ with $\left\langle \phi\right\rangle $
and $\left\langle \psi\right\rangle $ denoting the induced metrics
on the top Deligne pairing $\left\langle L,...,L\right\rangle $ in
the absolute case when $B$ is a point, then we have the following
``change of metric formula'': 
\[
\left\langle \phi\right\rangle -\left\langle \psi\right\rangle =\sum_{j=0}^{n}\int_{X}(\phi-\psi)(dd^{c}\phi)^{n-j}\wedge(dd^{c}\psi)^{j}
\]

\end{enumerate}
In order to define $\left\langle \phi\right\rangle $ for $\phi$
merely locally bounded, i.e. in $\mathcal{H}_{b}(L),$ we fix a reference
metric $\psi_{0}$ in $\mathcal{H}(L)$ and first set 
\begin{equation}
\mathcal{E}(\phi):=\mathcal{E}(\phi,\psi_{0}):=\frac{1}{(n+1)}\sum_{j=0}^{n}\int_{X}(\phi-\psi)(dd^{c}\phi)^{n-j}\wedge(dd^{c}\psi)^{j}\label{eq:def of energy}
\end{equation}
The functional $\mathcal{E}(\phi,\psi)$ is well-defined and finite
for any $\phi,\psi\in\mathcal{H}(L)_{b},$ using the Bedford-Taylor
product between the corresponding currents, and the functional $\mathcal{E}(\phi)$
on $\mathcal{H}_{b}(L)$ coincides with the restriction to $\mathcal{H}_{b}(L)$
of the functional $E$ in \cite[Section 1.4]{bbgz,bbegz} defined
on the whole space of singular metrics on $L$ with positive curvature
current. In particular, the functional $\mathcal{E}(\phi)$ is a primitive
for the one-form on $\mathcal{H}_{b}(L)$ defined by the Monge-Ampère
measure, i.e. 
\begin{equation}
\frac{d}{dt}_{|t=0}\mathcal{E}(\phi_{0}(1-t)+\phi_{1}t)=\int(\phi_{1}-\phi_{0})(dd^{c}\phi_{0})^{n}\label{eq:variational prop of energy}
\end{equation}
Now, for any $\phi\in\mathcal{H}_{b}(L)$ we can simply define the
corresponding metric $\left\langle \phi\right\rangle $ on the Deligne
pairing by
\[
\left\langle \phi\right\rangle :=\left\langle \psi\right\rangle +(n+1)\mathcal{E}(\phi,\psi)
\]
 for any fixed $\psi\in\mathcal{H}(L).$ It follows immediately from
the cocycle formula for $\mathcal{E}(\phi,\psi)$ (which in turn follows
from the variational property \ref{eq:variational prop of energy})
that $\left\langle \phi\right\rangle $ is independent of the choice
of $\psi$ and still satisfies the change of metric formula above,
i.e. 

\begin{equation}
\left\langle \phi\right\rangle -\left\langle \psi\right\rangle =(n+1)\mathcal{E}(\phi,\psi)\label{eq:change of metric formula as energy}
\end{equation}
Similarly, the first property 1 above also holds in the singular setting
of locally bounded metrics, by approximation, since the Bedford-Taylor
product is local and continuous under local decreasing limits. 
\begin{rem}
More generally, by the results in \cite{bbgz} the metric $\phi_{D}$
can be defined as long as $\phi_{0},...,\phi_{n}$ are in the finite
energy space $\mathcal{E}^{1}(X,L),$ but the locally bounded setting
above will be adequate for our purposes. 
\end{rem}
Let us now come back to the general setting of a test configuration
$\mathcal{L}\rightarrow\mathcal{X}\rightarrow\C$ for a polarized
variety $(X,L).$ Under appropriate regularity assumptions it was
shown in \cite{p-r-s} that the Donaldson-Futaki invariant of a test
configuration $(\mathcal{X},\mathcal{L})$ is the weight over $0:$
\begin{equation}
DF(\mathcal{X},\mathcal{L})=w_{0}(\eta)\label{eq:dfut as weight in text}
\end{equation}
 of the following $\Q-$line bundle over $\C:$ 
\begin{equation}
\eta:=\frac{1}{(n+1)L^{n}}\left(\mu\left\langle \mathcal{L},...,\mathcal{L}\right\rangle -(n+1)\left\langle -K_{\mathcal{X}/\C},\mathcal{L}...,\mathcal{L}\right\rangle \right)\,\,\,,\mu:=n(-K_{X})\cdot L^{n-1}/L^{n}\label{eq:def of eta}
\end{equation}
 where in the case when $L=-K_{X}$ we have $\mu=n.$ More precisely,
it was shown in \cite{p-r-s} that, up to natural isomorphisms, the
Knudson-Mumford expansion of the determinant line bundle $\det(\pi_{*}(kL))\rightarrow\C$
(with fibers $\det H^{0}(X_{\tau},kL_{\tau}))$ satisfies
\[
\det(\pi_{*}(kL))/kN_{k}=\frac{1}{(n+1)L^{n}}\left\langle \mathcal{L},...,\mathcal{L}\right\rangle -\frac{1}{k}\frac{1}{2}\eta+O(\frac{1}{k^{2}})
\]
 and $\eta$ is thus naturally isomorphic to Tian's CM-line bundle
\cite{p-t}. The proofs in \cite{p-r-s} were carried out under the
assumption that the total space $\mathcal{X}$ and the central fiber
$X_{0}$ be non-singular (in particular, there are no multiple fibers),
but as pointed out in \cite{p-r-s} the regularity assumptions can
be relaxed. Here we observe that the formula \ref{eq:dfut as weight in text}
holds as long as $\mathcal{X}$ is $\Q-$Gorenstein as a consequence
of the intersection theoretic expression for the Donaldson-Futaki
invariant recalled in the following proposition. The expression involves
the ``$\C^{*}-$equivariant compactification'' $\bar{\mathcal{L}}\rightarrow\bar{\mathcal{X}}\rightarrow\P^{1}$
of $\mathcal{L}\rightarrow\mathcal{X}\rightarrow\C$ obtained by using
the given $\C^{*}-$action $\rho$ on $(\mathcal{X},\mathcal{L})$
to identify $(\mathcal{X}^{*},\mathcal{L}_{|\mathcal{X}^{*}})$ with
$(X\times\C,p_{1}^{*}L),$ which trivially extends to $(X\times\P^{1}-\{0\},p_{1}^{*}L)$
so that $\bar{\rho}$ acts trivially over $\infty\in\P^{1}.$
\begin{prop}
\label{prop:wang}Let $(\mathcal{X},\mathcal{L},\rho)$ be a test
configuration (in particular, $\mathcal{X}$ is normal) and denote
by $\bar{\mathcal{L}}\rightarrow\bar{\mathcal{X}}\rightarrow\P^{1}$
the corresponding $\C^{*}-$equivariant compactification over $\P^{1}.$
Then\textup{ 
\begin{equation}
(n+1)L^{n}(DF(\mathcal{X},\mathcal{L}))=\mu\mathcal{\bar{\mathcal{L}}}\cdot\mathcal{\bar{\mathcal{L}}}\cdots\mathcal{\bar{\mathcal{L}}}+(n+1)K_{\bar{\mathcal{X}}/\P^{1}}\cdot\mathcal{\mathcal{\bar{\mathcal{L}}}}\cdots\mathcal{\mathcal{\bar{\mathcal{L}}}},\label{eq:df as intersection}
\end{equation}
where $K_{\bar{\mathcal{X}}/\P^{1}}$ is the relative canonical divisor
(viewed as a Weil divisor). }
\end{prop}
The previous formula is shown in \cite[Proposition 17]{w} under the
assumption that $\mathcal{X}$ is $\Q-$Gorenstein, so that $K_{\mathcal{X}/\C}$
is well-defined as a $\Q-$line bundle and in \cite[Theorem 3.2]{od}
for $\mathcal{X}$ normal (or more generally, Gorenstein in codimension
one. See also \cite[Proposition 6]{l-x} for a simple direct proof. 
\begin{prop}
\label{pro:(Phong-Ross-Sturm)-:-The}Let $(\mathcal{X},\mathcal{L})$
be a test configuration such that $\mathcal{X}$ is $\Q-$Gorenstein.
Then $DF(\mathcal{X},\mathcal{L})=w_{0}(\eta),$ where $\eta$ is
the $\Q-$line bundle defined by formula \ref{eq:def of eta}.\end{prop}
\begin{proof}
If $K_{\bar{\mathcal{X}}}$ is well-defined as a $\Q-$Cartier divisor
(i.e. as a $\Q-$line bundle) then it follows from the previous proposition
and the standard push-forward formula that 
\[
(n+1)L^{n}(DF(\mathcal{X},\mathcal{L}))=\int_{\P^{1}}\mu\pi_{*}(\mu(c_{1}(\mathcal{\bar{\mathcal{L}}})^{n+1})+(n+1)c_{1}(K_{\bar{\mathcal{X}}/\P^{1}})\wedge c_{1}(\mathcal{\bar{\mathcal{L}}})^{n}),
\]
 which, according to \ref{eq:deg of deligne is inters} coincides
with the degree of the corresponding sum $\bar{\eta}\rightarrow\P^{1}$
of extended Deligne pairings. But, this is nothing but the weight
over $0$ of the $\C^{*}-$action on $\eta_{|\tau=0}.$ Indeed, if
$F$ is a line bundle over $\P^{1}$ equipped with a $\C^{*}-$action
covering the standard action on $\P^{1}$ (viewed as a compactification
of $\C^{*})$ acting trivially on the line $F_{|\infty},$ then 
\begin{equation}
w_{0}(F)=\deg(F),\label{eq:weight as degree}
\end{equation}

\end{proof}

\subsubsection{A singular Kempf-Ness formula for weights involving the Lelong number}

Next we give a generalization of the Kempf-Ness type formula for the
weight appearing in Geometric Invariant Theory \cite{k-n}, which
is essentially equivalent to Lemma 6 in \cite{p-r-s}. Its formulation
involves the classical notion of a \emph{Lelong number} $l_{0}(\Phi)$
at zero of a subharmonic function $\Phi$ on the unit-disc in $\C,$
which may be defined as the sup over all numbers $\lambda$ such that
$\Phi(\tau)\leq\lambda\log|\tau|^{2}$ close to $\tau=0$ (equivalently,
$l_{0}(\Phi)$ is the mass of curvature current of $\Phi$ at the
origin: 
\begin{equation}
l_{0}(\Phi)=\int_{\{0\}}(dd^{c}\Phi).\label{eq:lelong as mass}
\end{equation}

\begin{lem}
\label{lem:Hilbert-Mumford}Let $F$ be a line bundle over the unit-disc
$\Delta$ in $\C$ equipped with a $\C^{*}-$action $\rho$ compatible
with the standard one on $\Delta$ and fix an $S^{1}-$invariant metric
$\Phi$ on $F$ with positive curvature current. Then the weight $w_{0}$
of the $\C^{*}$ action on the complex line $F_{0}$ is given by the
following formula involving right derivatives: 
\begin{equation}
w_{0}=-\lim_{t\rightarrow\infty}\frac{d}{dt}\log\left\Vert \rho(\tau)s_{1}\right\Vert _{\Phi}^{2}+l_{0}(\Phi)\label{eq:w as deric}
\end{equation}
 for $t=-\log|\tau|^{2}$ and $s_{1}$ a fixed element in the complex
line $F_{1}.$\end{lem}
\begin{proof}
This can be proved exactly as in Lemma 6 in \cite{p-r-s}, using that
$l_{0}(\phi)=-\lim_{t\rightarrow\infty}\frac{d}{dt}\phi(e^{-t/2})$
if $\phi$ is subharmonic on $\Delta$ and $S^{1}-$invariant. Alternatively,
a highbrow proof can be given as follows, using the equivariant compactification
$\bar{F}\rightarrow\P^{1},$ as in the discussion preceeding Proposition
\ref{prop:wang} (and by extending $\Phi$ to a metric on $\bar{F}\rightarrow\P,$
smooth close to $\infty\in\P^{1}):$ the section $s_{\tau}:=\rho(\tau)s_{1}$
defines a trivializing holomorphic section of $\bar{F}\rightarrow\C$
and hence, setting $v(\tau):=-\log\left\Vert \rho(\tau)s_{1}\right\Vert _{\Phi}^{2}$
on $\C^{*}$ we can decompose $\deg\bar{F}=\int_{\P^{1}}dd^{c}\Phi=\int_{\C^{*}}dd^{c}v+\int_{\{0\}}dd^{c}\Phi.$
But $\int_{\C}dd^{c}v(\tau)=\int_{-\infty}^{\infty}d(\frac{dv(e^{-t/2})}{dt})=\lim_{t\rightarrow\infty}dv(t)/dt-0,$
which concludes the proof using formulae \ref{eq:lelong as mass}
and \ref{eq:weight as degree}. 
\end{proof}

\subsection{\label{sub:The-Monge-Amp=0000E8re-equation}The Monge-Ampère equation
on $\mathcal{X}$ and geodesic rays}

Next we explain how to attach a canonical metric on the line bundle
$\mathcal{L}\rightarrow\mathcal{X}$ over a test configuration to
a given metric $\phi$ on $L\rightarrow X$ and the relation to weak
geodesic rays. This builds on ideas introduced in the work of Phong-Sturm
\cite{p-s,p-s1b} and Chen-Tang \cite{ch}.

Let $(X,L)$ be a polarized normal variety and $(\mathcal{X},\mathcal{L},\rho)$
a test configuration for $X$ (recall that $\mathcal{X}$ is assumed
normal). Denote by $M$ the variety with boundary obtained by restricting
$\mathcal{X}$ to the unit-disc $\Delta\subset\C.$ Given a locally
bounded metric $\phi_{1}$ with positive curvature on $L$ we let
$\phi$ be the metric on $\mathcal{L}\rightarrow M$ defined as the
following envelope: 
\begin{equation}
\phi:=\sup\{\psi:\,\,\,\psi\leq\phi_{1}\,\mbox{on\,\ensuremath{\partial M\}}}\label{eq:def of envelope}
\end{equation}
where $\psi$ ranges over all locally bounded metrics with positive
curvature form on $\mathcal{L}\rightarrow M$ and $\phi_{1}$ is identified
with the $S^{1}-$invariant metric on $\partial M$ induced by the
given metric (since we are not a priori assuming that $\psi$ is continuous
the boundary condition above means that, locally, $\limsup_{z_{i}\rightarrow z}\psi(z_{i})\leq\phi_{1}(z)$
for any sequence $z_{i}$ approaching a boundary point $z).$ Occasionally,
we will use the logarithmic real coordinate $t=-\log|\tau|^{2}$ on
the punctured disc $\Delta^{*}.$ We note that since $X$ is identified
with the fiber $X_{1}$ of $\mathcal{X}$ we can use the action $\rho$
to identify the metrics $\phi_{\tau}$ on $X_{\tau}$ with a curve
of metric 
\begin{equation}
\phi^{t}:=\rho(\tau)^{*}\phi_{\tau},\,\,\,\,\, t:=-\log|\tau|^{2}\label{eq:def of geodesic ray as pull-back}
\end{equation}
on $L.$ Next we will show that the metric $\phi$ above can be seen
as a solution to a Dirichlet problem for the Monge-Ampère operator
on $M.$ In fact, it will be convenient to formulate the result for
any test configuration:
\begin{prop}
\label{prop:reg for ma-eq}Let $(\mathcal{X},\mathcal{L})$ be a test
configuration for the polarized variety $(X,L)$ with normal total
space $\mathcal{X}.$ Then the following holds:
\begin{itemize}
\item $\phi$ is $S^{1}-$invariant
\item $\phi$ is locally bounded with positive curvature current and upper
semi-continuous in $M$
\item $\phi_{\tau}\rightarrow\phi_{1}$ uniformly as $|\tau|\rightarrow1$
(with respect to any fixed trivializing of $\mathcal{L}$ close to
a given boundary point). 
\item In the interior of $M$ we have that $(dd^{c}\phi)^{n+1}=0$ in the
sense of pluripotential theory.
\end{itemize}
\end{prop}
\begin{proof}
The first point follows immediately from the extremal defining of
$\phi.$ It will be convenient to identify the metric $\phi_{1}$
with a $\C^{*}-$invariant metric on $\mathcal{L}$ over the punctured
unit-disc $\Delta^{*}$ using the action $\rho.$ We will also, abusing
notation slightly, identify the coordinate $\tau$ with the psh function
on $\pi^{*}\tau$ on $\mathcal{X}.$ Let us first construct a\emph{
barrier,} i.e. a continuous metric $\tilde{\phi}$ on \emph{$\mathcal{L}$}
with positive curvature current such that $\tilde{\phi}=\phi_{1}$
on $\partial M$ and $\tilde{\phi}_{\tau}\rightarrow\phi_{1}$ as
$|\tau|\rightarrow1.$ To this end first observe that for $\epsilon>0$
sufficiently small there exist a continuous metric $\phi_{U}$ with
positive curvature on \emph{$\mathcal{L}\rightarrow U$} over the
open set $U:=\{|\tau|\leq\epsilon\}\subset\mathcal{X}.$ Indeed, we
can set $\phi_{U}=\phi_{FS}$ for the Fubini-Study metric induced
by a fixed embedding of $\mathcal{X}$ (see the end of section \ref{sub:K-polystability-and-test}).
Finally, we set $\tilde{\phi}:=\max\{\phi_{1}+\log|\tau|,\phi_{U}-C\}$
for $C$ sufficiently large so that $\tilde{\phi}=\phi_{U}-C$ for
$|\tau|$ sufficiently small and $\tilde{\phi}=\phi_{1}+\log|\tau|$
for $|\tau|>\epsilon/2.$ Since $\tilde{\phi}$ is a candidate for
the sup defining $\phi$ we conclude that 
\begin{equation}
\phi\geq\tilde{\phi}\geq\phi_{1}+\log|\tau|\label{eq:lower bound on env}
\end{equation}
Next, let us show that $\phi$ is locally bounded from above or equivalently
that there exists a constant $C'$ such that 
\begin{equation}
\phi\leq\phi_{FS}+C'\label{eq:upper bound on env}
\end{equation}
Accepting this for the moment we deduce that the envelope $\phi$
is finite with positive curvature current. Moreover, the upper bound
also implies that the upper semi-continuous regularization $\phi^{*}$
of $\phi$ is a candidate for the sup defining $\phi,$ forcing $\phi=\phi^{*}$
in the interior of $M,$ i.e. $\phi$ is upper semi-continuous there.
To prove the previous upper bound we note that since any candidate
$\psi$ for the sup defining $\phi$ satisfies $\psi\leq\phi_{FS}+C$
on the set $E:=\partial M$ it follows from general compactness properties
of positively curved metrics (or more generally, $\omega-$psh functions)
that there is a constant $C'$ such that $\psi\leq\phi_{FS}+C'$ on
all of $M.$ Indeed, by a simple extension argument we may as well
assume that $u:=\psi-\phi_{FS}$ extends as an $\omega-$psh function
to some compactification $\hat{\mathcal{X}}$ of $\mathcal{X}$ for
some semi-positive form current $\omega$ with continuous potentials.
But since $u\leq C$ on the non-pluripolar set $E$ it then follows
from Cor 5.3 in \cite{g-z} that $u\leq C'$ on all of $\hat{\mathcal{X}}$
(strictly speaking the variety $\hat{\mathcal{X}}$ is assumed non-singular
in \cite{g-z}, but we may as well deduce the result by pulling back
$u$ to a smooth resolution of $\hat{\mathcal{X}}).$ Alternatively,
$u$ can be shown to be bounded from above by using the maximum principle
to bound it by a solution to a Dirichlet type problem for the Laplace
operator with respect to a fixed Kähler metric on a resolution of
$M$ (compare the argument for the upper bound in \cite{p-s1b}). 

Let us next consider the behavior of $\phi$ on $\Delta^{*}$ by identifying
$\phi_{\tau}$ with $\phi^{t}$ as above for $t\in[0,\infty[.$ Since
$\phi$ is positively curved and $S^{1}-$invariant it follows that
$\phi^{t}$ is convex in $t$ on $]0,\infty[$ and in particular the
right derivative $\dot{\phi}(t)$ with respect to $t$ exist and define
an increasing function on $]0,\infty[.$ Hence, $\dot{\phi}(t)\leq C_{1}:=\dot{\phi}(t_{1})$
as $t\rightarrow0.$ Combined with the lower bound \ref{eq:lower bound on env}
this means that there exists a constant $C_{T}$ such that $\left|\dot{\phi}\right|\leq C_{T}$
for any $t\in[0,T]$ and thus $|\phi^{t}-\phi^{t'}|\leq C_{T}|t-t'|$
for $t$ and $t'$ sufficiently small positive numbers. But then it
follows that $\phi^{t}$ converges uniformly to the correct boundary
values as $t\rightarrow0$ (again using the lower bound \ref{eq:lower bound on env}).

As for the final point, the vanishing of the Monge-Ampère measure
$(dd^{c}\phi)^{n+1}$ on the regular part of the interior of $M$
is a standard local argument, which follows from comparison with the
solution of the homogenuous Monge-Ampère equation on small balls.
Since, the Monge-Ampère measure on a locally bounded metric does not
charge pluripolar sets and in particular not the singular locus of
$M$ this concludes the proof.
\end{proof}
According to the previous proposition the envelope $\phi$ thus induces
a \emph{weak geodesic ray} $\phi^{t}$ (formula \ref{eq:def of geodesic ray as pull-back})
in the space $\mathcal{H}_{b}(X,L)$ of all bounded positively curved
metrics, starting at a given metric (compare \cite{p-s}). For much
more precise regularity results (given suitably smooth data on $\partial M)$
expressed on a smooth resolution of $\mathcal{X}$ we refer to the
paper \cite{p-s1b} and to \cite{ch}. However, the point here is
that the modest regularity results above will be adequate for our
purposes and that are valid for any given locally bounded positively
curved metric $\phi_{1}.$
\begin{example}
\label{ex:product ma eq and vector field}Given $(X,L)$ and a metric
$\phi_{L}$ in $\mathcal{H}_{b}(L)$ we set $(\mathcal{X},\mathcal{L},\pi)=(X\times\C,p_{1}^{*}L,p_{2})$
and equip $\mathcal{L}$ with the metric $\phi:=p_{1}^{*}\phi_{L}\in\mathcal{H}_{b}(\mathcal{L}).$
Then $\phi$ is the unique solution to the Dirichlet problem for the
complex Monge-Ampère equation on $\pi^{-1}(\Delta)$ with boundary
data $p_{1}^{*}\phi_{L}.$ This example fits into the setting above
if one equips $(\mathcal{X},\mathcal{L})$ with the ``trivial''
action $\rho_{triv}$ covering the standard $\C^{*}-$action on $\C$
and then the ray $\phi_{t},$ as defined by formula \ref{eq:def of geodesic ray as pull-back},
is constant in $t:$ $\phi^{t}=\phi_{L},$ since $\rho_{triv}$ preserves
$\phi.$ On the other hand, if we are given a non-trivial holomorphic
vector field $V$ generating a $\C^{*}-$action on $(X,L),$ such
that the corresponding $S^{1}-$action preserves $\phi_{L},$ then
we can endow $(\mathcal{X},\mathcal{L})$ with the corresponding non-trivial
action $\rho$ (in the sense of Example \ref{ex:product test}), which
still preserves the boundary data. The corresponding ray $\phi_{t}$
determined by $\phi$ and $\rho$ is then given by $\phi_{t}=(e^{tV})^{*}\phi_{L}.$
Note that $\rho_{triv}=\rho\circ\rho_{(X,-V)}$ in the notation of
Example \ref{ex:product test} and hence $\phi$ is invariant under
$\rho\circ\rho_{(X,-V)}$ in the sense that $(\rho\circ\rho_{(X,-V)})^{*}\phi=\phi.$ 
\end{example}

\section{\label{sec:Singularity-structure-of}Proofs of the main results}

Recall that the Ding functional introduced in \cite{di}, in the setting
of smooth Fano manifolds $X,$ is the functional on the space of all
smooth positively curved metrics on $-K_{X}$ defined, in our notation,
by
\begin{equation}
\mathcal{D}(\phi):=-\frac{1}{(-K_{X})^{n}}\mathcal{E}(\phi)+\log\int_{X}e^{-\phi}.\label{eq:def of ding functional}
\end{equation}
It follows immediately from the variational property \ref{eq:variational prop of energy}
of $\mathcal{E}$ that the critical points of $\mathcal{D}$ are Kähler-Einstein
metrics. More generally, the functional $\mathcal{D}(\phi)$ is well-defined
and finite on the space $\mathcal{H}_{b}(-K_{X})$ of bounded metrics
on $-K_{X}$ as long as $X$ has log terminal singularities (see Lemma
\ref{lem:singular complex sing is lct} for the finiteness of the
integral piece). Moreover, for any curve $\phi_{t}$ in $\mathcal{H}_{b}(-K_{X})$
such that $\phi_{0}$ is a (singular) Kähler-Einstein metric and the
right derivative $\frac{d\phi_{t}}{dt}_{|t=0^{+}}$exists we have
\begin{equation}
\frac{d\mathcal{D}(\phi_{t})}{dt}_{|t=0^{+}}\geq0\label{eq:ke is almost critical pt}
\end{equation}
as follows from the affine concavity of $\mathcal{E}$ and the Kähler-Einstein
equation (see \cite[Formula 6.5]{bbgz}).

\subsubsection{\label{sub:Sketch-of-the}Sketch of the proofs of Theorem \ref{thm:df=00003Dding}
and Theorem \ref{thm:k-poly intro} in the $\Q-$Gorenstein case}

First note that in the case when the total space $\mathcal{X}$ of
the test configuration is $\Q-$Gorenstein and $\mathcal{L}=-K_{\mathcal{X}/\C}$
(as is, for example, the case for a special test configuration) the
line bundle $\eta\rightarrow\C,$ whose weight $w_{0}(\eta)$ over
zero is equal to $DF(\mathcal{X},\mathcal{L})$ (Section \ref{sub:Deligne-pairings-and}),
is simply given by the top Deligne pairing $-\frac{1}{(-K_{X})^{n}(n+1)}\left\langle -K_{\mathcal{X}/\C},...,-K_{\mathcal{X}/\C}\right\rangle .$
\footnote{When making the identifications between $\mathcal{L}=-K_{\mathcal{X}/\C}$
we have to take the lift $\rho$ of the $\C^{*}-$action on $\mathcal{X}$
to be the canonical lift to $-K_{\mathcal{X}/\C}$ (compare the beginning
of Section \ref{sub:Special-test-configurations}). %
}Given a locally bounded metric $\phi$ on $\mathcal{L}\rightarrow\C$
we define the corresponding \emph{Ding metric} $\Phi$ on $\eta$
as the induced Deligne metric $\left\langle \phi\right\rangle $ plus
the function 
\begin{equation}
v_{\phi}(\tau):=-\log\int_{X_{\tau}}e^{-\phi_{\tau}}\label{eq:def of v phi}
\end{equation}
 (using, as before, additive notation for metrics) on $\C^{*},$ which,
as will be explained in the next section extends to a unique subharmonic
function on $\C.$ By the change of metrics formula \ref{eq:change of metric formula as energy}
the Ding metric $\Phi(\tau)$ may, as the name suggest, be identified
with the Ding functional $\mathcal{D}(\phi^{t})$ (formula \ref{eq:def of ding functional})
along the corresponding ray of metric $\phi^{t}$ on $-K_{X}.$ More
generally, when $\mathcal{X}$ is $\Q-$Gorenstein and $\mathcal{L}$
is a general polarization for $\mathcal{X}\rightarrow\C$ the idea
is to replace $\eta$ with another line bundle $\delta,$ that we
shall call the\emph{ Ding line bundle,} defined by 
\[
\delta:=-\frac{1}{L^{n}(n+1)}\left\langle \mathcal{L},...,\mathcal{L}\right\rangle +\pi{}_{*}(\mathcal{L}+K_{\mathcal{X}/\C})\rightarrow\C,
\]
 where $\pi{}_{*}(\mathcal{L}+K_{\mathcal{X}/\C})$ is the corresponding
adjoint direct image sheaf. In this more general case we define the
corresponding \emph{Ding metric} $\Phi$ as the Deligne metric $\left\langle \phi\right\rangle $
plus the natural $L^{2}-$type metric on $\pi{}_{*}(\mathcal{L}+K_{\mathcal{X}/\C})$
induced by $\phi.$ Even if $\delta$ is, in general not isomorphic
to $\eta$ we will show that 
\[
\left(DF(\mathcal{X},\mathcal{L})=\right)w_{0}(\eta)=w_{0}(\delta)+p,\,\,\,\, p\geq0
\]
 where $p=0$ iff $\mathcal{L}$ is isomorphic to $-K_{\mathcal{X}/\C}.$
Applying the Kempf-Ness type Lemma \ref{lem:Hilbert-Mumford} to $w_{0}(\delta)$
then gives 
\begin{equation}
DF(\mathcal{X},\mathcal{L})=\lim_{t\rightarrow\infty}\frac{d}{dt}\mathcal{D}(\phi^{t})+q,\,\,\,\, q=p+l_{0},\label{eq:formula for df in sketch}
\end{equation}
where $l_{0}$ is the Lelong number $l_{0}$ at zero of the Ding metric,
which in turn is shown to coincides with the non-negative Lelong number
of the $L^{2}-$type metric on $\pi{}_{*}(\mathcal{L}+K_{\mathcal{X}/\C}).$
As a consequence the error term $q$ is non-negative. Finally, taking
$\phi^{t}$ to be the geodesic ray associated to $(\mathcal{X},\mathcal{L})$
emanating from a given Kähler-Einstein metric and using convexity
properties of $\mathcal{D}(\phi^{t})$ then proves Theorem \ref{thm:k-poly intro}.
In the case of a general test configuration $(\mathcal{X},\mathcal{L})$
the idea is to apply the previous argument to a log resolution of
$\mathcal{X}.$ 
\begin{rem}
\label{Rem:li-x}It should be pointed out that, by a recent result
of Li-Xu \cite{l-x} which uses recent advances in MMP, a Fano variety
$X$ is K-polystable iff it is K-polystable for all \emph{special}
test configurations and hence, if one invokes \cite{l-x}, it is enough
to consider special test configurations in the proof of Theorem \ref{thm:k-poly intro}.
Anyway, as indicated above the proof in the general case is not much
more involved and combined with the existence results in \cite{c-d-s}
it yields an analytical/differential geometric proof of the result
of Li-Xu referred to above, at least for $X$ smooth (since it is
shown in \cite{c-d-s} that if a Fano manifold $X$ is K-polystable
for all special test configuration, then $X$ admits a Kähler-Einstein
metrics). We recall that an important ingredient in the proof in \cite{l-x}
is to show that $DF(\mathcal{X},\mathcal{L})$ decreases along a relative
MMP (first without and then with scaling) which, after an initial
base change, modifies $(\mathcal{X}_{0},\mathcal{L}_{|\mathcal{X}_{0}})$
so that $(i)$ $(\mathcal{X},\mathcal{X}_{0})$ has log canonical
singularities and $(ii)$ $\mathcal{L}$ is isomorphic to $-K_{\mathcal{X}/\C}.$
Interestingly, as we will show (independently of \cite{l-x}) that
$(i)$ holds iff $l_{0}=0$ and $(ii)$ holds iff $p=0$ and thus
the error term $q$ in formula \ref{eq:formula for df in sketch}
achieves it minimal value $0$ iff $(\mathcal{X},\mathcal{L})$ is
a test configuration of the form produced by the MMP procedure in
\cite{l-x}.
\end{rem}

\subsection{\label{sub:The-Ding-line}The Ding line bundle and the Ding metric}

Let $(\mathcal{X},\mathcal{L})$ be a test configuration for a Fano
variety $(X,-K_{X})$ and fix an equivariant log resolution $p:\,\mathcal{X}'\rightarrow\mathcal{X}$
of $(\mathcal{X},\mathcal{X}_{0})$ and write $\mathcal{L}':=p^{*}\mathcal{L}.$
Then $(\mathcal{X}',\mathcal{L}'$) is a semi-test configuration for
$(X,-K_{X}).$ First assume, to fix ideas, that the original Fano
variety $X$ is smooth with $\mathcal{L}$ a line bundle over $\mathcal{X}$
and define a the \emph{Ding line bundle }$\delta'\rightarrow\C$ by
\begin{equation}
\delta':=-\frac{1}{L^{n}(n+1)}\left\langle \mathcal{L}',...,\mathcal{L}'\right\rangle +\pi'_{*}(\mathcal{L}'+K_{\mathcal{X}'/\C})\rightarrow\C,\label{eq:ding line bundle}
\end{equation}
(when $X$ is smooth the direct image sheaf $\pi'_{*}(\mathcal{L}'+K_{\mathcal{X}'/\C})$
is indeed a line bundle, as explained below). Given a metric $\phi$
on $\mathcal{L}\rightarrow\mathcal{X}$ we denote by $\Phi'$ the\emph{
generalized Ding metric} on $\delta',$ defined as the Deligne metric
on the top Deligne pairing of $\mathcal{L}$ twisted by the $L^{2}-$metric
on $\pi'_{*}(\mathcal{L}'+K_{\mathcal{X}'/\C}),$ induced by $\phi':=p^{*}\phi.$
Note that in general $\mathcal{L}$ is only assumed to be a $\Q-$line
bundle, i.e. $r\mathcal{L}$ is a line bundle for some positive integer
$r$ and then we may simply define $\pi'_{*}(\mathcal{L}'+K_{\mathcal{X}'/\C}):=\pi'_{*}(r(\mathcal{L}'+K_{\mathcal{X}'/\C})/r$
as a $\Q-$line bundle (which is easily seen to be independent of
$r)$ and let $\Phi'$ be the metric defined by the corresponding
$L^{2/r}-$norm 
\[
\left\Vert s_{\tau}\right\Vert _{L_{\phi}^{2/r}}:=\left(\int_{\mathcal{X}_{\tau}}|s_{\tau}|^{2/r}e^{-\phi'_{\tau}}\right)^{r/2},
\]
 where we have identified the restriction $s_{\tau}$ of $s\in H^{0}(X',\mathcal{L}'+K_{\mathcal{X}'/\C})$
with a holomorphic $(n,0)-$form on $X_{\tau}$ with values in $\mathcal{L}'_{|X_{\tau}}$
(compare the notation in Section \ref{sub:canonical measures}).

Turning to the case of a general Fano variety $X$ with log terminal
singularities first recall that, since the variety $\mathcal{X}^{*}(:=\mathcal{X}-\mathcal{X}_{0})$
has log terminal singularities, we have $p^{*}K_{\mathcal{X}}=K_{\mathcal{X}'}+D^{*}$
on $\mathcal{X}'-\mathcal{X}'_{0}$ for a (sub) klt $\Q-$divisor
$D^{*},$ whose closure in $\mathcal{X}'$ we will denote by $D.$
We can decompose $D=D'-E'$ as a difference of effective $\Q-$divisors
where $E'$ has integral coefficients (but we are not claiming that
the $D'$ and $E'$ have no common components) . We may and will assume
that the log resolution is such that the support of $D$ has simple
normal crossings and is transversal to $\mathcal{X}'_{0}.$ We then
define 
\[
\delta':=-\frac{1}{L^{n}(n+1)}\left\langle \mathcal{L}',...,\mathcal{L}'\right\rangle +\pi'_{*}(\mathcal{L}'+D'+K_{\mathcal{X}'/\C})\rightarrow\C,
\]
and denote by $\Phi'$ the corresponding metric on $\delta',$ which
is defined using the log adjoint $L^{2/r}-$metric on $\pi'_{*}(\mathcal{L}'+D'+K_{\mathcal{X}'/\C})$
 
\[
\left\Vert s_{\tau}\right\Vert _{L_{\phi}^{2/r}}:=\left(\int_{\mathcal{X}_{\tau}}|s_{\tau}'|^{2/r}e^{-(\phi'_{\tau}+\phi_{D'})}\right)^{r/2}
\]
To see that $\pi'_{*}(\mathcal{L}'+D'+K_{\mathcal{X}'/\C})$ is indeed
a line bundle over $\C$ first note that over $\C^{*},$ where the
sheaf is globally free, any fiber may be identified with $H^{0}(X',E'),$
where $X'=p^{*}X$ and $E'$ is $p-$exceptional, so that $\dim H^{0}(X',E')=1.$
The extension property to all of $\C$ then follows from general principles.
Indeed, the direct image sheaf is clearly torsion-free and since the
base is a curve any torsion-free sheaf is automatically locally free
(indeed, to get a local generator for the sheaf close to $0\in\C$
one simply takes an element with minimal vanishing order at $0).$

\subsection{Positivity/continuity properties of the Ding metric}

In this section we assume given a Fano variety $X$ (with log terminal
singularities) and a locally bounded metric $\phi_{1}$ on $-K_{X}$
with positive curvature current. Fixing a test configuration $(\mathcal{X},\mathcal{L})$
for $(X,-K_{X})$ and a resolution $(\mathcal{X}',\mathcal{L}'),$
as in the previous section, we denote by $\phi$ the induced $S^{1}-$invariant
locally bounded metric on $\mathcal{L}\rightarrow M(\subset\mathcal{X})$
(compare Section \ref{sub:The-Monge-Amp=0000E8re-equation}), by $\phi^{t}$
the corresponding weak geodesic in $\mathcal{H}_{b}(-K_{X})$ and
by $\Phi$ the corresponding Ding type metric on the Ding line bundle
$\delta'\rightarrow\Delta.$ 

The study of the positivity properties of the Ding metric relies on
the following fundamental positivity result of Berndtsson-Paun for
direct image vector bundles (applied the rank one case):
\begin{lem}
\label{lem:(Berndtsson-Paun)-Let-}Let $\mathcal{X}$ be a non-singular
projective variety with a morphism $\pi:\,\mathcal{X}\rightarrow\C$
which is smooth (i.e. a submersion) over $\C^{*}$ and a $\Q-$line
bundle $\mathcal{L}\rightarrow\mathcal{X}$ equipped with with a singular
metric $\phi$ with positive curvature and such that the restriction
$\phi_{\tau}$ to each fiber $X_{\tau},$ for $\tau\in\C^{*},$ satisfies
$e^{-\phi_{\tau}}\in L_{loc}^{1}.$ If $\pi_{*}(\mathcal{L}+K_{\mathcal{X}/\C})\rightarrow\C$
is defined as a $\Q-$line bundle (i.e. if $\dim H^{0}(\mathcal{L}_{|X_{\tau}}+K_{\mathcal{X}_{\tau}})=1$
for $\tau\in\C^{*})$ then the corresponding $L^{2/r}-$metric on
the line bundle $\pi_{*}(\mathcal{L}+K_{\mathcal{X}/\C})\rightarrow\C$
has positive curvature in the sense of currents (where $r$ is a positive
integer such that $r\mathcal{L}$ is a line bundle). \end{lem}
\begin{proof}
The positivity over $\C^{*}$ is a special case of the main results
in \cite{bern1,be-p} (note that by assumption the $L^{2/r}-$metric
is finite over $\C^{*}).$ The positivity over a neighborhood of $0$
also follows from the arguments in \cite{be-p}. But as the latter
positivity was not stated explicitly in \cite{be-p} we provide a
detailed proof%
\footnote{The positivity in question is also a special case of the very general
positivity results in \cite[Theorem 1.1]{p-tak}, which appeared during
the revision of the present paper, whose proof uses among other things,
semi-stable reduction.%
}. Fix a local trivializing section of $\pi_{*}(r(\mathcal{L}+K_{\mathcal{X}/\C}))\rightarrow\C$
over a small neighborhood $V$ of $0\in\C.$ It may be identified
with a global holomorphic section $s$ of $r(\mathcal{L}+K_{\mathcal{X}/\C})\rightarrow\mathcal{X}_{|V}$
with the property that $\tau$ does not divide $s.$ Fix a local coordinate
$\tau$ on $\C$ and let 
\begin{equation}
v(\tau):=-\frac{1}{r}\log\left\Vert s\right\Vert _{L_{\phi}^{2/r}}^{2}(=-\log\left\Vert s\right\Vert _{L_{\phi}^{2/r}}^{2/r})\label{eq:definition of weight v}
\end{equation}
 be the corresponding local weight of the $L^{2/r}-$metric on the
$\Q-$line bundle $\frac{1}{r}\pi_{*}(r(\mathcal{L}+K_{\mathcal{X}/\C})).$
By basic properties of subharmonic functions the positivity in question
is equivalent to an upper bound on $v(\tau)$ on $V$ or equivalently
a lower bound 
\begin{equation}
\left\Vert s\right\Vert _{L_{\phi}^{2/r}}^{2/r}:=\int_{\mathcal{X}_{\tau}}|s_{\tau}|^{2/r}e^{-\phi_{\tau}}\geq\epsilon>0,\,\,\,\tau\in V\label{eq:lower bound in pos}
\end{equation}
where we have identified $|s|^{2/r}e^{-\phi}$ with a family of measures
over $\mathcal{X}^{*}:=\pi^{-1}(\C^{*})$ (as in Section \ref{sub:canonical measures}).
A subtle point is that the assumptions of the lemma do not exclude
that the holomorphic section $s$ vanishes identically on the reduction
of $\mathcal{X}_{0}$ (i.e. on the underlying variety); for example,
this can happen if $\mathcal{X}_{0}$ has components with different
multiplicities. On the other hand it follows from a local application
of the generalized Ohsawa-Takegoshi extension theorem in \cite[Lemma 1.1]{be-p}
that there exists a uniform constant $C$ (i.e. independent of $\tau)$
such that 
\[
\int_{V}|s|^{2/r}e^{-\phi}\leq C\int_{\mathcal{X}_{\tau}}|s_{\tau}|^{2/r}e^{-\phi_{\tau}},\,\,\,\tau\in V
\]
But since $s$ is non-vanishing over $V-\{0\}$ this implies the desired
lower bound \ref{eq:lower bound in pos}.
\end{proof}
In particular, by the previous proposition the function $v_{\phi}(\tau)$
is subharmonic, i.e $v_{\phi}(e^{-t/2})$ is convex, as long as $\phi$
is psh (as already observed in \cite{bbgz,bern2} for $X$ non-singular
and in \cite{bbegz} in general). Next, we will show that if the curvature
of the line bundle $\pi_{*}(\mathcal{L}+K_{\mathcal{X}/\C})$ vanishes
identically, then - in the presence of a $\C^{*}-$action as in the
definition of a test configuration - $\mathcal{X}$ has to be a product.
This will be crucial when considering the case $DF(\mathcal{X},\mathcal{L})=0$
in the proof of Theorem \ref{thm:k-poly intro}. 
\begin{prop}
\label{prop:flat vector bundle}Let $X$ be a Fano variety with log
terminal singularities and $(\mathcal{X},\mathcal{L})$ a test configuration
for $(X,-K_{X})$ such that $\mathcal{X}$ is $\Q-$Gorenstein and
$\mathcal{L}=-K_{\mathcal{X}/\C}.$ Assume that $\mathcal{L}$ is
equipped with an $S^{1}-$invariant locally bounded metric $\phi$
with positive curvature current such that the induced curvature current
of the direct image sheaf $\pi_{*}(\mathcal{L}+K_{\mathcal{X}/\C})$
vanishes identically on $\C$ (or more generally, over some neighbourhood
of $0\in\C).$ Then $\mathcal{X}$ is isomorphic to $X\times\C.$ \end{prop}
\begin{proof}
First recall that it was shown in \cite[Theorem 6.1]{bern2} in the
case of $X$ smooth and \cite[Theorem 5.1]{bbegz} in the general
case, that if $v_{\phi}(\tau)$ is harmonic, for $\tau\in\Delta^{*}$
(i.e. $v_{\phi}(e^{-t})$ is affine in $t$), then there is a family
of biholomorphic maps $F^{t}$ indexed by $t\in\R$ such that 
\begin{equation}
(F^{t})^{*}dd^{c}\phi^{t}=dd^{c}\phi^{0},\label{eq:pull back relation on kahler metrics}
\end{equation}
 where $\phi^{t}$ denotes the ray of metric on $-K_{X}$ corresponding
to $\phi,$ i.e. $\phi^{t}=\rho(\tau)^{*}\phi_{\tau},$ for $t=-\log|\tau|^{2}$
(compare Section \ref{sub:The-Monge-Amp=0000E8re-equation}). Moreover,
as shown in \cite[Section 4.1]{bern2} $F^{t}=\mbox{exp \ensuremath{(t\mbox{Re}V)}}$for
a holomorphic $(1,0)-$vector field $V$ on $X$ such that the flow
of its imaginary part $\mbox{Im}V$ preserves $dd^{c}\phi^{0}.$ In
fact, in the present setting we even have 
\begin{equation}
(F^{t})^{*}\phi^{t}=\phi^{0},\label{eq:pull-back relation on ray of metrics}
\end{equation}
where we have used the same notation $F^{t}$ for the canonical lift
of $F^{t}$ to $-K_{X}.$ To see this first note that the relation
\ref{eq:pull back relation on kahler metrics} implies that  $(F^{t})^{*}\phi^{t}=\phi^{0}+a(t)$
for some function $a(t)$ and since $v_{\phi}(e^{-t})=-\log\int_{X}e^{-\phi_{0}}+a(t)$
is assumed affine it follows that $a(t)$ is also affine, i.e. $a(t)=at+c$
for some real numbers $a$ and $c.$ But then it follows that $v_{\phi}(\tau)=-\log\int_{X}e^{-\phi_{0}}+a\log|\tau|^{2}+c$
and hence $a$ is equal to the Lelong number $l_{0}$ of $v_{\phi}(\tau)$
at $0.$ Now, $l_{0}$ coincides with the mass at $0$ of the curvature
current on $\pi_{*}(\mathcal{L}+K_{\mathcal{X}/\C})$ (see formula
\ref{eq:lelong as mass}) which is assumed to vanish and hence $a=0$
and since $F^{0}$ is the identity this means that $c$ also vanishes,
which proves \ref{eq:pull-back relation on ray of metrics}. Next,
we make the following 
\begin{equation}
\mbox{Claim:\,}\mbox{\ensuremath{V\,\,\mbox{generates a \ensuremath{\C^{*}-}action \ensuremath{\rho_{(X,V)}}on \ensuremath{X}}}}\label{eq:claim periodic}
\end{equation}
(i.e. the flow of $\mbox{Im\ensuremath{V}}$ has period $2\pi).$
This is obvious if $X$ admits no non-trivial holomorphic vector fields
(since $V=0$ then) and in the general case the claim follows from
Lemma \ref{lem:vector field and c star} below, by observing that
$\mathcal{X}_{0}$ is reduced. Indeed, as observed above, the Lelong
number $l_{0}=0$ and hence it follows from Proposition \ref{prop:lelong of direct image}
below that $\mathcal{X}_{0}$ is indeed reduced.%
\footnote{In fact, one does not have to use the fact that $\mathcal{X}_{0}$
is reduced if one instead performs a base change followed by a normalization
to get a new normal test configuration $\mathcal{X}'$ with reduced
central fiber to which Lemma \ref{lem:vector field and c star} can
be applied; this corresponds to replacing $V$ with $mV$ for some
positive integer $m.$ %
} We will write $F_{\tau}:=\rho_{(X,V)}(\tau)$ for the family of biholomorphic
maps on $X,$ indexed by $\tau\in\C^{*},$ generated by $V$ (so that
$F_{\tau}=F^{t}$ for $\tau=e^{-t/2}$). Next, we set 
\[
G_{\tau}:=\rho_{\tau}\circ F_{\tau}^{-1};\,\,\,\,\mathcal{X}_{0}\rightarrow\mathcal{X}_{\tau}
\]
so that 
\begin{equation}
G_{\tau}^{*}\phi_{\tau}=\phi_{1}\label{eq:pull-back relation for metrics phi for tau fix}
\end{equation}
Using an equivariant embedding into $\P^{N}$ (see section \ref{sub:K-polystability-and-test})
we can identify $G_{\tau}$ with a family of holomorphic embeddings
\[
G_{\tau}:\, X\rightarrow\P^{N},\,\,\,\, X_{\tau}:=G_{\tau}(X),\,\,\, G_{\tau}^{*}\mathcal{O}(1)=-K_{X},\,\,\,\tau\in\C^{*}
\]
Since $\phi$ is a locally bounded metric on $\mathcal{L}\rightarrow\mathcal{X}$
we have that $|\phi-p_{1}^{*}\phi_{FS}|\leq C$ over $\Delta,$ where
$\phi_{FS}$ denotes the Fubini-Study metric on $\mathcal{O}(1)\rightarrow\P^{N}$
and $p_{1}$ the natural projection $\P^{N}\times\Delta\rightarrow\P^{N}.$
Hence, \ref{eq:pull-back relation for metrics phi for tau fix} gives
\begin{equation}
\sup_{X}|G_{\tau}^{*}\phi_{FS}-\phi_{1}|\leq C\label{eq:uniform bound on map}
\end{equation}
 for $\tau\in\Delta^{*}.$ We claim that the corresponding holomorphic
map $G$ from $X\times\Delta^{*}$ to $\P^{N}\times\Delta$ extends
to a holomorphic map $X\times\Delta\rightarrow\P^{N}\times\Delta.$
To see this introduce local coordinates on an open set $U\subset X$
centered at a given point $x_{0}$ in $X$ and fix an affine piece
$\C^{N}$ of $\P^{N}$ such that $G_{1}(U)\subset\C^{N}.$ Then there
exists a bounded subset $B$ of $\C^{N+1}$such that $G_{\tau}(U)\subset B$
for any $\tau\in\C^{*}.$ Indeed, by the very definition of the Fubini-Study
metric $\phi_{FS}$ the bound \ref{eq:uniform bound on map} gives
that there exists a constant $C'$ such that $|G_{\tau}(z)|\leq C'$
for any $z\in U$ and $\tau\in\C^{*}$ and hence $B$ can be taken
as as ball of radius $C'$. Hence, applying Harthog's extension theorem
(coordinatewise) thus gives that $G_{|U\times\Delta^{*}}$ extends
to a unique holomorphic map of $U\times\Delta$ into $\C^{N}\times\Delta$
and since the point $x_{0}$ was arbitrary this proves the claim.
Moreover, since $\mathcal{X}$ is normal and in particular closed
and irreducible, it follows that $G$ maps $X\times\Delta$ surjectively
onto $\mathcal{X}\subset\P^{N}\times\Delta.$ We claim that $G$ is
a finite map. Since $G$ is, by construction, injective on $X\times\Delta^{*}$
it will be enough to prove that the restriction of $G$ to $X\times\{0\},$
that we denote by $G_{0},$ defines a finite map from $X$ to $\P^{N}.$
To this end we note that that $G_{0}$ pulls back the cohomology class
$c_{1}(\mathcal{O}(1))$ on $\P^{N}$ to $c_{1}(L)(=c_{1}(-K_{X})).$
Indeed, by construction 
\begin{equation}
G_{\tau}^{*}c_{1}(\mathcal{O}(1))=c_{1}(L)\label{eq:pullback relation on cohom}
\end{equation}
 for any $\tau\in\C^{*}$ and since $G_{\tau}\rightarrow G_{0}$ as
$\tau\rightarrow0$ (in the sense established above) it follows that
\ref{eq:pullback relation on cohom} also holds for $\tau=0$ and
hence $G_{0}$ pulls back a Kähler class on $\P^{n}$ to a Kähler
class on $X.$ But then $G_{0}$ has to be finite, since otherwise
$G_{0}$ would contract some $p$- dimensional subvariety $V_{p}$
of $X.$ This would mean that the $p$ th intersection number of $c_{1}(L)$
with $V_{p}$ vanishes, contradicting the fact that $c_{1}(L)$ is
a Kähler class. All in all this means that $G$ defines a finite birational
morphism from $X\times\C$ to $\mathcal{X}$ and since $\mathcal{X}$
is assumed normal it then follows from Zariski's Main Theorem that
$G$ is a biholomorphism, as desired. The same argument also applies
if $\tau=1$ is replaced with any $\tau_{0}\neq0$ such that $v_{\phi}(\tau)$
is harmonic for $|\tau|\leq|\tau_{0}|.$ 
\end{proof}
The previous proposition can be seen as a partial generalization to
singular fibrations of a result in \cite{bern1b} (see Theorem 1.2
and the discussion in Section 4.1 in \cite{bern1b} where the fibration
is assumed to be a submersion, but without any assumptions on $\C^{*}-$equivariance).
In the proof we used the following lemma of independent interest:
\begin{lem}
\label{lem:vector field and c star}Assume given a non-trivial holomorphic
vector field $V$ on a normal variety $X$ of type $(1,0)$ with a
fixed lift to $L\rightarrow X$ and a metric $\phi^{0}\in\mathcal{H}_{b}(L),$
which is invariant under the flow $\mbox{Im}V.$ Then $V$ generates
a $\C^{*}-$action on $X$ iff there exists a (normal) test configuration
$(\mathcal{X},\mathcal{L},\pi,\rho)$ with reduced central fiber $\mathcal{X}_{0}$
such that $(e^{t\mbox{Re}V})^{*}\phi^{0}=\rho(\tau)^{*}\phi_{\tau}$,
where $\phi_{\tau}=\phi_{|\mathcal{X}_{\tau}}$ for some $\phi\in\mathcal{H}_{b}(\pi^{-1}(\Delta),\mathcal{L}).$ \end{lem}
\begin{proof}
If $V$ generates a $\C^{*}-$action then the test configuration can
be taken as a product, as explained in Example \ref{ex:product ma eq and vector field}.
But in order to prove the converse, which is what was used in the
proof of Proposition \ref{prop:flat vector bundle}, we will have
to show that the total space $\mathcal{X}$ of the given test configuration
is necessarily a product. We will continue with the notation from
the Proposition \ref{prop:flat vector bundle}. First observe that
the map $G_{\tau}$ above is well-defined for $\tau=e^{-t/2}\in[0,1[$
and set $G^{t}:=G_{e^{-t}}.$ By the argument above the map $G_{0}:=\lim_{t_{j}\rightarrow\infty}G^{t}$
still exists and defines a holomorphic map from $X$ to $\mathcal{X}\subset\P^{N},$
if one uses ``normal families'' instead of Harthog's extension theorem,
for some subsequence $t_{j}\rightarrow\infty$ (but we are not claiming
that $G_{0}$ is independent of the subsequence at this point). 

\emph{Step 1:} $\mathcal{X}_{0}$ is reduced and irreducible, i.e.
defined by a variety $X_{0}$ and $G_{0}$ is finite and generically
one-to-one, mapping $X_{1}$ onto $X_{0}$ and 
\begin{equation}
G_{0}^{*}\phi_{0}=\phi_{1}\label{eq:pull-back relation of metrics phi in the limit}
\end{equation}
where $\phi_{1}$ is restricted metric on $\mathcal{L}_{|X_{0}}.$ 

To prove the first point we decompose the central fiber $\mathcal{X}_{0},$
viewed as a divisor on the normal variety $\mathcal{X},$ in its irreduicble
components: $\mathcal{X}_{0}=\sum_{i=1}^{p}m_{i}E_{i},$ where $E_{i}$
are distinct prime divisors on $\mathcal{X}$ (i.e. reduced and irreducible).
Since $X$ is assumed irreducible it follows that $G_{0}$ maps $X_{1}$
onto one of the components of $\mathcal{X}_{0}$ that we may take
to be the one labeled by $i=1.$ By the definition of $G_{0}$ we
have the following convergence in the sense of currents on $\P^{N}$
\[
\lim_{t_{j}}[X_{t_{j}}]=(G_{0})_{*}[X]=d[E_{1}],
\]
 where $d$ is the degree of the surjective finite map $G_{0}:X\rightarrow E_{1}.$
But the lhs above is also equal to $[\mathcal{X}_{0}]$ (by basic
convergence properties of currents, or using the Chow variety) and
hence $d[E_{1}]=\sum_{i=1}^{p}m_{i}[E_{i}],$ which forces $d=m_{1}$
and $p=1.$ Now, $\mathcal{X}_{0}$ was assumed reduced and hence
$d=m_{1}=1,$ which implies (by basic properties of the degree) that
$G_{0}$ is a finite generally one-to-one map from $X_{1}$ onto the
irreducible variety $X_{0}(=E_{1}).$ Next, to prove \ref{eq:pull-back relation of metrics phi in the limit}
we fix a point $x_{0}\in X_{0}\cap\mathcal{X}_{reg}\cap(X_{0})_{reg}$
(i.e. $x_{0}\in X_{0}-Z$ where $Z$ has codimension one in $X_{0},$
since $\mathcal{X}$ is normal). Then there exists a neighbourhood
$U$ of $x_{0}$ in $\mathcal{X}$ with holomorphic coordinates centered
at $x_{0}$ of the form $(z,\tau).$ The relation \ref{eq:pull-back relation of metrics phi in the limit}
then follows from \ref{eq:pull-back relation for metrics phi for tau fix}
and a simple continuity argument at $x_{0}.$ This means that the
two psh metrics $G_{0}^{*}\phi_{0}$ and $\phi_{1}$ on $L\rightarrow X$
coincide on the Zariski open subset $G_{0}^{-1}(X_{0}-Z)$ of $X$
and hence everywhere, by the local identity principle for psh functions. 

\emph{Step 2: }The ``pull-back'' $\rho'$ to $X$ of the restricted
action of $\rho$ to $X_{0}$ coincides with the flow of $V$

Denote by $Y^{\nu}$ the normalization of an irreducible variety $Y$
and by $\nu$ the normalization map $\nu:Y^{\nu}\rightarrow Y,$ which
is a finite generically injective morphism. By the universal property
of the normalization $G_{0}$ lifts to $X_{0}^{\nu}$ and hence the
lifted map $G_{0}^{\nu}$ satisfies $G_{0}^{\nu*}\nu^{*}\phi_{0}=\phi_{1}.$
Moreover, since $X$ is normal it follows from Zariski's main theorem
that $G_{0}$ is an isomorphism and hence $X$ isomorphic to the normalization
of $X_{0}$ and $G_{0}$ may be identified with $\nu.$ Using the
universal property of the normalization again this means that the
the pull-back $G_{0}^{*}\rho$ which is a priori only well-defined
on a Zariski open subset of $X$ where $G_{0}$ is holomorphically
invertible, extends to give a well-defined holomorphic $\C^{*}-$action
$\rho'$ on $X.$ To prove that $\rho'$ is generated by $V$ we will
use a (singular) Hamiltonian formalism. To a given pair $(\psi,W)$
consisting of a locally bounded psh metric $\psi$ on a line bundle
$L\rightarrow Y$ over a complex variety $Y$ and a holomorphic vector
field $W$ on $Y$ with a fixed lift to $L,$ preserving $\psi,$
we associate a function $h_{(\psi,W)}$ on $Y,$ that we will call
the\emph{ Hamiltonian:} 
\[
h_{(\psi,W)}=\frac{d}{ds}_{|s=0}(e^{sW})^{*}\psi
\]
in the sense of right derivatives ($h_{(\psi,W)}$ exists and is finite
since $\psi$ is locally psh and hence $(e^{sW})^{*}\psi$ is convex
wrt $s).$ In the particular case when $W$ is the generator of a
$\C^{*}-$action $\rho$ we set $h_{(\psi,\rho)}:=h_{(\psi,W)}.$
We let $h,h_{0}$ and $h_{1}$ be the Hamiltonian functions on $\mathcal{X},X_{0}$
and $X_{1}$ corresponding to $(\rho,\phi),(\rho_{|X_{0}},\phi_{0})$
and $(V,\phi_{1}),$ respectively. Note that it follows directly from
the definition that $h_{|X_{0}}=h_{0}.$ Next we will show that

\begin{equation}
G_{0}^{*}h_{0}=h_{1}\label{eq:pull-back of hamilt}
\end{equation}
To this end first observe that
\begin{equation}
h_{1}(x)=h(G^{t}(x)).\label{eq:identifications of hamilt with t}
\end{equation}
 Indeed, under the isomorphism $X_{1}\times\C^{*}\rightarrow\mathcal{X}^{*},\,(x,\tau)\mapsto x_{\tau}:=\rho(\tau)x$
determined by $\rho$ the action $\rho$ on $\mathcal{X}^{*}$ may
be identified with the ``trivial'' action on $X_{1}\times\C^{*}$
generated by the vector field $\tau\frac{\partial}{\partial\tau}$
and the metric $\phi$ on $\mathcal{L}$ may be identified with the
metric on $p_{1}^{*}L$ suggestively written as $\phi(x,\tau):=\phi^{t}(x),$
where $t=-\log|\tau|^{2}.$ In the present setting we have, by assumption,
that $\phi_{t}=\exp(tV)^{*}\phi_{1}$ where $V$ also determines the
map $G^{t}$ from $X_{1}$ to $X_{e^{-t/2}}$ defined above, which
may be identified with the map $(x,1)\mapsto(\exp(-tV)x,e^{-t/2})\in X_{1}\times\C^{*}.$
Using these identifications we may write 
\[
h(G_{t}(x))=\frac{d}{ds}_{|s=0}\phi(\exp(-tV)x,e^{-(t+s)/2})=\frac{d}{ds}_{|s=0}\phi^{t+s}(\exp(-tV)x)=
\]
\[
=\frac{d}{ds}\phi_{1}(\exp(t+sV)\exp(-tV)x)=\frac{d}{ds}_{|s=0}\phi_{1}(\exp(sV)x)=:h_{1}(x),
\]
which proves \ref{eq:identifications of hamilt with t}. Finally,
setting $t=t_{i}$ and letting $t_{t}\rightarrow\infty$ gives, since
$G_{0}(x):=\lim_{t_{j}\rightarrow\infty}G^{t_{j}}(x)$ that $h_{1}(x)=h(G_{0}(x))$
and hence $h_{1}(x)=h_{0}(G_{0}(x)),$ proving \ref{eq:pull-back of hamilt}.
All in all, combining the pull-back relations \ref{eq:pull-back relation of metrics phi in the limit}
and \ref{eq:pull-back of hamilt} reveals that $h_{(\phi,V)}=h_{(\phi,G_{0}^{*}\rho}).$
But for a fixed metric $\phi\in\mathcal{H}_{b}(X)$ we have that $W\mapsto h_{(\phi,W)}$
is injective and hence $V$ is the generator of the $\C^{*}-$action
$G_{0}^{*}\rho$ which proves the claim \ref{eq:claim periodic}.
The injectivity used above is standard under the regularity assumption
that there exists a point $x\in X$ such that $dd^{c}\phi$ is smooth
and strictly positive close to $x,$ since $dd^{c}\phi(\mbox{Im}V,\cdot)=dh_{(\phi,V)},$
which may be inverted to determine $\mbox{Im \ensuremath{V}}$ and
hence $V.$ In the general case, the injectivity follows from the
general formalism in \cite{b-n} (or from Proposition 8.2 in \cite{bern2}).
Anyway, in the application to the proof of Theorem \ref{thm:k-poly intro}
$\phi_{1}$ will be a Kähler-Einstien metric and in particular the
regularity assumption above holds. \end{proof}
\begin{prop}
\label{prop:The-Ding-metric has pos prop}Let $(\mathcal{X},\mathcal{L})$
be a test configuration for a Fano variety $(X,-K_{X})$ with log
terminal singularities. Then the Ding metric associated to a weak
geodesic ray $\phi^{t}$ as above has the following positivity properties: 
\begin{itemize}
\item Its curvature defines a positive current on $\Delta$ (and in particular
the function $\mathcal{D}(\phi^{t})$ is convex in $t)$
\item If $\mathcal{X}$ is $\Q-$Gorenstein, \emph{$\mathcal{L}=-K_{\mathcal{X}/\C}$}
and the curvature current of the Ding metric vanishes on some disc
centered at $0$ (i.e. $\mathcal{D}(\phi^{t})$ is affine on $]T,\infty[$
and the Lelong number $l_{0}$ of the Ding metric vanishes) then $\mathcal{X}$
is a product test configuration.
\end{itemize}
\end{prop}
\begin{proof}
First we note that the curvature of the Deligne metric $\left\langle \phi\right\rangle $
on $\left\langle \mathcal{L},...,\mathcal{L}\right\rangle $ is non-negative
if $\phi$ is psh and vanishes if the corresponding ray $\phi^{t}$
is a weak geodesic, as follows from the push-forward formula \ref{eq:curvature of the deligne pairing}
Alternatively, since $\left\langle \phi\right\rangle $ is locally
bounded from above (by the continuity result in Prop \ref{prop:The-Ding-metric is cont})
it is enough to consider the holomorphically trivial case over $\Delta^{*}$
where the result amounts to a well-known property of the functional
$\mathcal{E}$ (see \cite{bbegz}). Combined with the positivity in
the previous lemma this shows that the Ding metric has positive curvature
current. More precisely, in the case when $X$ is singular we apply
the previous lemma to the line bundle $p^{*}\mathcal{L}+D'\rightarrow\mathcal{X}$
equipped with the metric $p^{*}\phi+\phi_{D'},$ where $\phi_{D'}$
is the singular psh metric on the line bundle $\mathcal{O}(D')$ induced
by $D'$ which satisfies $e^{-\phi_{D'}}\in L_{loc}^{1},$ since $D'$
is klt. The last point follows immediately from the previous proposition.\end{proof}
\begin{prop}
\label{prop:The-Ding-metric is cont}The Ding metric associated to
a weak geodesic ray is continuous on $\Delta^{*}$ up to the boundary
circle.\end{prop}
\begin{proof}
Let us first verify that if $\phi$ is a locally bounded positively
curved metric on $\mathcal{L}\rightarrow\mathcal{X}$ then the Deligne
metric $\left\langle \phi\right\rangle $ on $\left\langle \mathcal{L},...,\mathcal{L}\right\rangle \rightarrow\C$
is locally bounded on $\Delta$ and continuous at the boundary of
$\Delta.$ To this end we first recall that if $\psi$ is a smooth
metric on $\mathcal{L}$(i.e. the restriction to $\mathcal{X}$ of
a smooth metric) then it was shown by Moriwaki \cite[Theorem A]{mo}
that the corresponding Deligne metric $\left\langle \psi\right\rangle $
on the top Deligne product on $\left\langle \mathcal{L},...,\mathcal{L}\right\rangle \rightarrow\C$
is continuous. But since $\phi$ is a locally bounded metric on $\mathcal{L}$
we have that $u:=\phi-\psi$ is a bounded function on $\mathcal{X}$
and hence it follows from the change of metric formula \ref{eq:change of metric formula as energy}
that 
\[
\left|\left\langle \phi\right\rangle -\left\langle \psi\right\rangle \right|\leq c^{1}(L)^{n}\sup_{\mathcal{X}}|u|
\]
is bounded (where $L$ denotes the restriction of $\mathcal{L}$ to
a generic fiber). Hence $\left\langle \phi\right\rangle $ is locally
bounded, as desired. Alternatively, the local boundedness of $\left\langle \phi\right\rangle $
can be verified directly by induction over the relative dimension,
using the recursive definition of $\left\langle \phi\right\rangle $
\cite{zh}. Similarly, the continuity at $\tau=1$ follows from continuity
properties at $\tau=1$ of $\phi_{\tau}.$ Indeed, by Prop \ref{prop:reg for ma-eq}
we have that $\phi_{\tau}\rightarrow\phi_{1}$ uniformly as $\tau\rightarrow1,$
i.e. $\phi^{t}\rightarrow\phi^{0}$ and hence it follows from the
change of metrics formula that, in a fixed local trivialization close
to $\tau=1,$ we have 
\[
\left|\left\langle \rho(\tau)\phi_{\tau}\right\rangle -\left\langle \phi_{1}\right\rangle \right|\leq c(L)^{n}\sup_{X}|(\rho(\tau)^{*}\phi_{\tau})-\phi_{1}|\rightarrow0,
\]
 as $\tau\rightarrow1$ and moreover $v_{\phi}(\tau)\rightarrow v_{\phi}(1).$
This shows in particular that, over $\Delta^{*},$ the Ding metric
$\Phi(=\left\langle \phi\right\rangle +v_{\phi})$ may be identified
with a locally bounded $S^{1}-$invariant convex function (by the
previous proposition) which is continuous up to $\partial\Delta.$ 
\end{proof}

\subsection{\label{sub:Singularity-structure-of}Singularity structure of the
Ding metric }

We continue with the setup and notation in the previous section. We
will give a detailed description of the singularity of the Ding metric
at $\tau=0$ (which however is not used in the proof of Theorem \ref{thm:k-poly intro}).
The key point is the observation that the Lelong number $l_{0}$ at
$0$ of $L^{2/r}$- metrics as above can be expressed in terms the
log canonical thresholds of the corresponding central fiber. First
recall that the \emph{complex singularity exponen}t $c_{x}(v)$ at
a point $x$ in a complex manifold $X$ of a local psh function $v$
(i.e. defined on some neighborhood $U_{x}$ of $x)$ is defined by
\[
c_{x}(v):=\sup_{c\in\R}\left\{ c:\,\exists U_{x}\,\, e^{-cv}\in L^{1}(U_{x},dV)\right\} ,
\]
where $dV$ is a local volume form. When 
\begin{equation}
v=v_{D}:=\sum_{i}a_{i}\log|f_{i}|^{2}\,\,\,\, D:=\sum_{i}a_{i}D_{i}\label{eq:def of v D}
\end{equation}
where $f_{i}$ is a local holomorphic function determining a zero
prime divisor $D_{i}$ the number $c_{x}(v)$ coincides with the l\emph{og
canonical threshold at $x,$} denoted by$c_{x}(D),$ of the $\Q-$divisor
$D:$ 
\begin{equation}
c_{x}(v)=c_{x}(D)\label{eq:complex sing is lct of div}
\end{equation}
 (see \cite[Proposition 8.2]{ko}). The latter number admits a purely
algebraic definition valid for any log pair\emph{ }$(X,D),$ i.e.
without assuming $X$ non-singular: 
\[
c_{x}(D):=\sup_{c\in\R}\left\{ c:\, cD\,\mbox{is\,\ log canonical\ close\ to\,\ensuremath{x}}\right\} ,
\]
(compare Section \ref{sub:Fano-varieites-and}). More generally, given
a log pair $(X,\Delta)$ and an effective $\Q-$Cartier divisor $D$
on $X$ the \emph{log canonical threshold of $(X,\Delta,D),$ along
$Z,$} may be defined by

\[
c_{Z}(X,\Delta,D):=\sup_{c\in\R}\left\{ c:\,\Delta+cD\,\mbox{is\ lc\,\ close to\,\ensuremath{Z}}\right\} 
\]
\cite[Definition 8.1]{ko}. In particular, by definition, $(X,D)$
is lc iff $c_{X}(X,0,D)\geq1.$ It will be convenient to introduce
the following analytic counterpart of $c_{Z}(X,\Delta,D)$ obtained
by replacing $D$ with a psh function $v$ defined in a neighborhood
of $Z:$ 
\[
c_{Z}(X,\Delta,v):=\sup_{c\in\R}\left\{ c:\,\exists U_{Z}\,\, e^{-cv}\in L_{loc}^{1}(U_{Z},\mu_{(X,\Delta;\phi_{0})})\right\} ,
\]
 where $\phi_{0}$ is a fixed locally bounded metric on $-(K_{X}+\Delta)$
and $\mu_{(X,\Delta;\phi_{0})}$ denotes the corresponding measure
on $X$ (see Section \ref{sub:canonical measures}). By the boundedness
assumption on $\phi_{0}$ the definition above is independent of the
choice of $\phi_{0}.$ More generally, $v$ can be taken as a metric
on a $\Q-$line bundle $L\rightarrow X.$ The following generalization
of the identity \ref{eq:complex sing is lct of div} holds: 
\begin{lem}
\label{lem:singular complex sing is lct}Let $(X,\Delta)$ be a log
pair. For $v=v_{D}$ as in formula \ref{eq:def of v D} we have that 

\[
c_{Z}(X,\Delta,v)=c_{Z}(X,\Delta,D)
\]
In particular, $(X,\Delta)$ has log terminal singularities iff $\mu_{(X,\Delta;\phi_{0})}$
gives finite volume to $X.$\end{lem}
\begin{proof}
This is essentially well-known, but for completeness we provide a
proof (see \cite{d-k,ko} for the standard case when $X$ is smooth
and $\Delta$ is trivial and \cite[Lemma 6.8]{egz}\cite[Lemma 3.2]{bbegz}
for the last statement of the lemma). First note that it follows directly
from the definitions that it will be enough to show that $e^{-v}\in L_{loc}^{1}(U_{Z},\mu_{(X,\Delta;\phi_{0})})$
iff $(X,\Delta+D)$ is klt (since we may then replace $v$ by $cv$
and $D$ by $cD$ and take the sup with respect to $c).$ To this
end take a log resolution of $(X,\Delta+D)$ and denote by $\pi$
the corresponding morphism from $X'$ to $X.$ Denote by $\Delta'$
the divisor on $X'$ such that $(X',\Delta')$ corresponds to $(X,\Delta)$
as in formula \ref{eq:pull-back of log can}. Since $D$ is $\Q-$Cartier
this means that $(X',\Delta'+\pi^{*}(D))$ corresponds to $(X,\Delta+D).$
Next, using $\mu_{(X,\Delta;\phi_{0})}=\mu_{(X',\Delta';\pi^{*}\phi_{0})}$
gives that $v\in L_{loc}^{1}(U_{Z},\mu_{(X,\Delta;\phi_{0})}$ iff
$I:=\int_{\pi^{-1}(U_{Z})}e^{-((v_{\Delta'+\pi^{*}(D)})-v_{0})}dV<\infty$
(after perhaps shrinking $U_{Z})$ for some volume form $dV$ on $X,$
where $v_{0}$ is a fixed locally bounded metric on the $\Q-$line
bundle $\mathcal{O}(\Delta'+\pi^{*}(D)).$ But $X'$ is smooth and
$\Delta'+\pi^{*}(D)$ has simple normal crossings and hence it follows
from the basic fact that $c_{0}(\log|z|^{2})=1$ in $\C$ and Fubini's
theorem that the the integral $I$ is finite iff the coefficients
of $\Delta'+\pi^{*}(D)$ are $<1$ (just as in \cite[Lemma 6.8]{egz}\cite[Lemma 3.2]{bbegz})
which equivalently means that $(X,\Delta+D)$ is klt, as desired. \end{proof}
\begin{prop}
\label{prop:lelong of direct image}Assume that $\mathcal{X}$ is
a normal $\Q-$Gorenstein variety and $\pi:\,\mathcal{X}\rightarrow\C$
a projective morphism over $\C$ which is smooth (i.e. a submersion)
over $\C^{*}.$ Let $\mathcal{L}\rightarrow\mathcal{X}$ be a semi-positive
$\Q-$line coinciding with $-K_{\mathcal{X}/\C}$ over $\C^{*}$ and
$\phi$ a locally bounded metric on $\mathcal{L}$ with positive curvature
current. Denote by $l_{0}$ the Lelong number $l_{0}$ at $\tau=0$
of the induced $L^{2/r}-$metric on the line bundle $\pi_{*}(\mathcal{L}+K_{\mathcal{X}/\C})\rightarrow\C.$
Then 
\begin{equation}
l_{0}=1-c_{\mathcal{X}_{0}}(\mathcal{X},-\Delta,\mathcal{X}_{0})\label{eq:lelong in terms of lct}
\end{equation}
where $\Delta$ is the zero-divisor in $\mathcal{X}$ of any local
trivialization section of $\pi_{*}(\mathcal{L}+K_{\mathcal{X}/\C})\rightarrow\C,$
identified with an element of $H^{0}(U,\mathcal{L}+K_{\mathcal{X}/\C}).$
Moreover, denoting by $E_{i}$ the reduced components of $\mathcal{X}_{0}$
we define the numbers $m_{i}$ and $c_{i}$ by 
\[
\mathcal{X}{}_{0}=\sum_{i}m_{i}E_{i},\,\,\,\Delta'=\sum_{i}c_{i}E_{i},
\]
 the following holds:
\begin{itemize}
\item If $\mathcal{X}$ is smooth and the central fiber $\mathcal{X}_{0}$
has simple normal crossings, then 
\begin{equation}
l_{0}=\max_{i}\frac{m_{i}-1-c_{i}}{m_{i}}\label{eq:lelong as max}
\end{equation}

\item If $\mathcal{L}=-K_{\mathcal{X}/\C},$ then $l_{0}=0$ iff $(\mathcal{X},\mathcal{X}_{0})$
is log canonical near $\mathcal{X}_{0}$ iff $\mathcal{X}_{0}$ is
reduced and the normalization of $\mathcal{X}_{0}$ has log canonical
singularities. 
\end{itemize}
\end{prop}
\begin{proof}
Fix a local trivializing section of $\pi_{*}(r(\mathcal{L}+K_{\mathcal{X}/\C}))\rightarrow\C$
over a neighborhood $V$ of $0\in\C$ identified with a global holomorphic
section $s$ of $r(\mathcal{L}+K_{\mathcal{X}/\C})\rightarrow\mathcal{X}_{|V}$
as in the proof of Lemma and denote by $v(\tau)$ the corresponding
weight on $V$ (formula \ref{eq:definition of weight v}). By Lemma
\ref{lem:(Berndtsson-Paun)-Let-} $v(\tau)$ is subharmonic and we
denote by $l_{0}$ the Lelong number of $v$ at $\tau=0$ (as in Lemma
\ref{lem:Hilbert-Mumford}). It will be very useful to represent the
Lelong number $l_{0}$ as follows 
\begin{equation}
l_{0}=\inf\left\{ l:\,\int_{V}e^{-(v(\tau)+(1-l)\log|\tau|^{2})}id\tau\wedge d\bar{\tau}<\infty\right\} \label{eq:lelong as inf}
\end{equation}
(the equivalence with the ordinary definition follows immediately
from $c_{0}(\log|\tau|^{2})=1).$ Next recall that $\nu_{\phi}:=|s|^{2/r}e^{-\phi}$
defines a measure on $\mathcal{X}_{|V},$ naturally attached to $\phi$
(see Section \ref{sub:canonical measures}). The measure $\nu_{\phi}$
has the property that for, any continuous function $g$ on $\C,$
\begin{equation}
\int_{\mathcal{X}_{|V}}\nu_{\phi}\pi^{*}g=\int_{V}e^{-v(\tau)}g(\tau)id\tau\wedge d\bar{\tau}\label{eq:lift of integral to total space}
\end{equation}
The proof of formula \ref{eq:lelong in terms of lct} is simply a
matter of unraveling definitions. First assume, to fix ideas, that
$\mathcal{X}$ is smooth. Then there exists a locally bounded metric
$\psi$ on $-K_{\mathcal{X}}$ and $\left\Vert \cdot\right\Vert $
on $\mathcal{L}$ such that 
\begin{equation}
\nu_{\phi}=\left\Vert s\right\Vert ^{2}\mu_{\psi}\label{eq:nu in terms of mu}
\end{equation}
where $\mu_{\psi}$ is the measure on $\mathcal{X}$ corresponding
to $\psi$ (see Section \ref{sub:canonical measures}). Indeed, fixing
holomorphic coordinates $w=(w_{0},...,w_{n})$ on $U\subset\mathcal{X}$
and a trivialization $s_{\mathcal{L}}$ of $\mathcal{L}\rightarrow U$
the section $s$ of $r(\mathcal{L}+K_{\mathcal{X}/\C}),$ restricted
to $U,$ may be written as $s=f_{U}s_{\mathcal{L}}\otimes dw\otimes\pi^{*}\frac{\partial}{\partial\tau},$
for a holomorphic function $f_{U}$ on $U$ and thus on $U$ 
\[
\nu_{\phi_{U}}=|f_{U}(w)|^{2/r}e^{-\phi_{U}(w)}i^{n^{2}}dw\wedge d\bar{w},
\]
 where, by assumption, $\phi$ is bounded on $U.$ This proves formula
\ref{eq:nu in terms of mu} in case $\mathcal{X}$ is smooth. More
generally, if $\mathcal{X}$ is $\Q-$Gorenstein, then $\nu_{\phi}$
and $\mu_{\psi}$ are still well-defined and by the argument above
above the relation \ref{eq:nu in terms of mu} holds on the regular
locus of of $\mathcal{X}$ and thus everywhere since the two measures
do not charge the singular locus of $\mathcal{X}.$ In particular,
combining formula \ref{eq:lift of integral to total space} (for $g(\tau)=e^{-(1-l)\log|\tau|^{2})})$
and formula \ref{eq:nu in terms of mu} gives 
\begin{equation}
l_{0}=\inf\left\{ l:\,\int_{\mathcal{X}_{|V}}e^{-(1-l)\log|\tau|^{2})}\left\Vert s\right\Vert ^{2}\mu_{\psi}<\infty\right\} =1-c_{\mathcal{X}_{0}}(\mathcal{X},-\Delta,\mathcal{X}_{0})\label{eq:form for lelong in proof}
\end{equation}
 using in the last equality Lemma \ref{lem:singular complex sing is lct}
(applied to the log pair $(\mathcal{X},-\Delta)$ where $\Delta$
is the zero-divisor of $s$ and with $v=\log|\tau|^{2}$ and $D=\mathcal{X}_{0})$
which concludes the proof of formula \ref{eq:lelong in terms of lct}.
Formula \ref{eq:lelong as max} then follows immediately from basic
formula for log canonical thresholds of simple normal crossing divisors.
For completeness we provide a proof: by assumption $\Delta+c\mathcal{X}_{0}$
has simple normal crossings and 
\[
-\Delta+c\mathcal{X}_{0}=\sum_{i}(-c_{i}+cm_{i})E_{i}
\]
 and since $c_{\mathcal{X}_{0}}(\mathcal{X},\Delta,\mathcal{X}_{0})$
is the sup over all $c$ such that the coefficients above are $\leq1$
we get $c_{\mathcal{X}_{0}}(\mathcal{X},\Delta,\mathcal{X}_{0})=\min_{i}\frac{1+c_{i}}{m_{i}}$
which, by formula \ref{eq:lelong in terms of lct}, proves the formula
in the first point. To prove the second point we apply formula \ref{eq:form for lelong in proof}
to the case where $\Delta=0$ and thus $l_{0}=0$ iff $1-c_{\mathcal{X}_{0}}(\mathcal{X},\mathcal{X}_{0})=0$,
i.e. iff $(\mathcal{X},\mathcal{X}_{0})$ if log canonical. Now, if
$(\mathcal{X},\mathcal{X}_{0})$ is log canonical then it follows
that $\mathcal{X}_{0}$ is reduced and, by adjunction, that its normalization
has log canonical singularities (see \cite[2.7]{al01}). Finally,
the converse follows from ``inversion of adjunction'', i.e. from
the main result of \cite{kawi}, previously conjectured by Shokurov
(the special case when $\mathcal{X}_{0}$ has log terminal singularities
follows from a previous result of Kollar et al \cite[Theorem 7.5]{ko}).
\end{proof}
Combining the last point in the previous proposition with Lemma \ref{lem:char of special test}
gives the following
\begin{cor}
\label{cor:vanshing of lelong for gen ding}Let $(\mathcal{X},\mathcal{L})$
be a test configuration (with a priori non-normal total space $\mathcal{X})$
for a smooth Fano manifold $(X,-K_{X})$ such that the central fiber
$\mathcal{X}_{0}$ is normal. Then $\mathcal{X}$ is $\Q-$Gorenstein
with $\mathcal{L}=-K_{\mathcal{X}/\C}$ and $l_{0}=0$ iff the variety
$X_{0}$ has log canonical singularities. In particular, if $\mathcal{X}$
is a special test configuration then $l_{0}=0.$
\end{cor}

\subsubsection{Interlude on Calabi-Yau degenerations}

Before continuing we make a brief detour to point out that Proposition
\ref{prop:lelong of direct image} also has some applicatiotions to
the non-Fano case when $\mathcal{X}$ is Gorenstein and $K_{\mathcal{X}}$
is a trivial line bundle and hence the generic fiber $X_{\tau}$ is
a Calabi-Yau manifold. Then $F:=\pi_{*}(K_{\mathcal{X}/\C})\rightarrow\C$
is the Hodge line bundle and its curvature $dd^{c}v(\tau)$ (wher
$v(\tau)$ is given by formula \ref{eq:definition of weight v}) coincides
with the Weil-Peterson metric $\omega_{WP}$ on the punctured base
$\C^{*}$ i.e. the pull-back of the Weil-Peterson metric on the moduli
space of Calabai-Yau manifolds (see \cite{w c-l} and references therein).
In this case $v(\tau)$ admits an expansion of the form $v(\tau)=l_{0}\log|\tau|^{2}+\beta\log(\left|\log|\tau|^{2}\right|^{-1})+O(1)$
as $\tau\rightarrow0,$ for some integer $\beta\in[0,n],$ where $O(1)$
denotes a term which is bounded in $C_{loc}^{2}$ \cite[Theorem 4.1]{w c-l}.
Accordingly, Proposition \ref{prop:lelong of direct image} applied
to this case says that $l_{0}=c(\mathcal{X},\mathcal{X}_{0})$ and
$l_{0}=0$ iff $v(\tau)$ has at worst log log singularities, i.e.
\begin{equation}
v(\tau)=\beta\log(\left|\log|\tau|^{2}\right|^{-1})+O(1)\label{eq:log log sing}
\end{equation}
iff $\mathcal{X}_{0}$ is reduced and its normalization has log canonical
singularities. This observation can be used to simplify the proof
of Theorem 1.2 in \cite{to} (which answers in the affirmative a question
of Wang) saying that \emph{if the Weil-Peterson metric $\omega_{WP}$
on the base of the Calabi-Yau fibration $\pi:\,\mathcal{X}^{*}\rightarrow\C^{*}$
as above is not complete as $\tau\rightarrow0,$ then, after a base
change, the central fiber $\mathcal{X}_{0}$ may be modified so that
$\mathcal{X}_{0}$ is reduced and has canonical singularities. }The
starting point is, following \cite{to}, the recent advances in the
MMP which give that, after a base change, one can assume that $(\mathcal{X},\mathcal{X}_{0})$
is relatively minimal (i.e. divisorially log terminal, dlt) and in
particular log canonical. Hence, by Prop \ref{prop:lelong of direct image}
$l_{0}=0$ i.e. $v(\tau)$ has at worst a log log singularity as in
formula \ref{eq:log log sing}. The incompleteness assumption on $\omega_{WP}=dd^{c}v=\beta\omega_{P}+O(1),$
where $\omega_{P}$ is the Poincaré form on $\C^{*}$ thus forces
$\beta=0$ and hence $v(\tau)$ is bounded as $\tau\rightarrow0.$
But then one concludes that $X_{0}$ is irreducible with log terminal
singularities (and hence canonical singularities since $K_{X}$ is
assumed Cartier), as desired (the last claim is the content of the
implication $(c)\implies(a)$ in Theorem 1.1 in \cite{to}, whose
proof is due to Sebastien Boucksom: by Fatou's lemma $i^{n^{2}}\int_{X_{0}}\Omega\wedge\Omega<\infty$
for a non-trivial $\Omega\in H^{0}(X_{0},K_{X_{0}})$ and adjunction,
using the dlt assumption, then implies that $X_{0}$ is irreducible
and normal and thus log terminal by Lemma \ref{lem:singular complex sing is lct}).

\subsection{Expressing the Donaldson-Futaki invariant in terms of the Ding functional }

Consider the following $\Q-$line bundle over $\C$ defined in terms
of the fixed log resolution:
\[
\eta':=-\frac{1}{(n+1)L^{n}}\left\langle \mathcal{L}',...,\mathcal{L}'\right\rangle +\frac{1}{L^{n}}\left\langle \mathcal{L}'+K_{\mathcal{X}'/\C}+D',\mathcal{L}'...,\mathcal{L}'\right\rangle ,
\]
(recall that $L=-K_{X}$ here so that $\mu=n$ in formula \ref{eq:def of eta}).
Then 
\begin{equation}
DF(\mathcal{X},\mathcal{L})=w_{0}(\eta')\label{eq:df on resol in terms of eta}
\end{equation}
Indeed, combining Prop \ref{prop:wang} with the push-forward formula
for intersection numbers gives 
\[
(n+1)L^{n}(DF(\mathcal{X},\mathcal{L}))=np^{*}(\mathcal{\bar{\mathcal{L}}})\cdot p^{*}(\mathcal{\bar{\mathcal{L}}})\cdots p^{*}(\mathcal{\bar{\mathcal{L}}})+(n+1)p^{*}K_{\overline{\mathcal{X}}/\P^{1}}\cdot p^{*}(\mathcal{\bar{\mathcal{L}}})\cdots p^{*}(\mathcal{\bar{\mathcal{L}}}),
\]
Now, $p^{*}K_{\overline{\mathcal{X}}/\P^{1}}$ is equal to $K_{\overline{\mathcal{X}'}/\P^{1}}+D'$
modulo the $p-$exceptional divisor $E',$ which give no contribution
to the intersection number above, since $p^{*}\mathcal{L}$ is trivial
on $E'.$ Formula \ref{eq:df on resol in terms of eta} then follows
precisely as in the proof of Prop \ref{pro:(Phong-Ross-Sturm)-:-The}. 
\begin{lem}
\label{lem:df bigger than ding weight}We have that $DF(\mathcal{X},\mathcal{L})=w_{0}(\eta')\geq w_{0}(\delta')$\end{lem}
\begin{proof}
Using $DF(\mathcal{X},\mathcal{L})=w_{0}(\eta')$ and decomposing
\begin{equation}
\eta'=\delta'+\left(\frac{1}{L^{n}}\left\langle K_{\mathcal{X}'/\C}+D'+\mathcal{L}',\mathcal{L}',...,\mathcal{L}'\right\rangle -\pi'_{*}(\mathcal{L}'+K_{\mathcal{X}'/\C}+D')\right)\label{eq:eta as sum of delta' + correction}
\end{equation}
 reveals that it is enough to show that $w_{0}\left(\frac{1}{L^{n}}\left\langle K_{\mathcal{X}'/\C}+D'\text{+}\mathcal{L}',\mathcal{L}',...,\mathcal{L}'\right\rangle -\pi'_{*}(\mathcal{L}'+K_{\mathcal{X}'/\C})+D'\right)=$
\[
=\frac{1}{L^{n}}(K_{\bar{\mathcal{X}'}/\P^{1}}+D'+\bar{\mathcal{L}'})\cdot\bar{\mathcal{L}'}\cdots\bar{\mathcal{L}'}-\deg\pi'_{*}(\bar{\mathcal{L}'}+K_{\bar{\mathcal{X}'}/\P^{1}}+D')\geq0,
\]
 where we have used the the compactification $\bar{\mathcal{X}'}$
of the resolution $\mathcal{X}'$ and the corresponding extension
$\bar{\mathcal{L}'}$ of $\mathcal{L}'$ in the first equality (together
with formula \ref{eq:weight as degree}). To simplify the notation
we consider the case when $X$ is smooth so that $D'=0,$ but the
general case is essentially the same. Note that the formula above
involving the degrees is invariant under $\mathcal{L}'\rightarrow\mathcal{L}'\otimes\pi'^{*}\mathcal{O}_{\P^{1}}(m)$
and hence we may as well assume that $\deg\pi'_{*}(\bar{\mathcal{L}'}+K_{\bar{\mathcal{X}'}/\P^{1}})=0$
(this corresponds to a performing an overall twisting of the original
action $\rho$ on $\mathcal{L}).$ But the latter vanishing means
that the line bundle $\pi'_{*}(\bar{\mathcal{L}'}+K_{\bar{\mathcal{X}'}/\P^{1}})\rightarrow\P^{1}$
admits a global trivializing holomorphic section $s,$ unique up to
scaling by a non-zero complex constant. In particular, $s$ induces
a global holomorphic section $\bar{\mathcal{L}'}+K_{\bar{\mathcal{X}'}/\P^{1}}\rightarrow\bar{\mathcal{X}'}.$
This means that $\bar{\mathcal{L}'}+K_{\bar{\mathcal{X}'}/\P^{1}}$
is linearly equivalent to an effective divisor $E$ (whose support
is contained in the central fiber). But then it follows, since $\bar{\mathcal{L}'}$
is relatively semi-ample, that 
\begin{equation}
(K_{\bar{\mathcal{X}'}/\P^{1}}+\bar{\mathcal{L}'})\cdot\bar{\mathcal{L}'}\cdots\bar{\mathcal{L}'}=E\cdot\bar{\mathcal{L}'}\cdots\bar{\mathcal{L}'}\geq0\label{eq:intersection in terms of e}
\end{equation}
 which thus concludes the proof.
\end{proof}
Now we are ready to prove the following more precise version of Theorem
\ref{thm:DF=00003Dding intro}, stated in the introduction:
\begin{thm}
\label{thm:df=00003Dding}Let $X$ be a Fano variety with log terminal
singularities and \emph{$(\mathcal{X},\mathcal{L})$ a test configuration
(with normal total space) for $(X,-K_{X})$ with $\phi$ denoting
a locally bounded metric on $\mathcal{L}\rightarrow\mathcal{X}\rightarrow\Delta$
with positive curvature current. Then, setting $\phi^{t}:=\rho(\tau)^{*}\phi_{\tau},$
identified with a ray of metrics on $-K_{X}$ we have }
\begin{equation}
DF(\mathcal{X},\mathcal{L})=\lim_{t\rightarrow\infty}\frac{d}{dt}\mathcal{D}(\phi^{t})+q,\label{eq:df in terms of ding in thm}
\end{equation}
 where $q$ is a non-negative rational number determined by the polarized
central fiber $(\mathcal{X}_{0},\mathcal{L}_{|\mathcal{X}_{0}})$
with the following properties, in the case that $X$ is smooth:
\begin{itemize}
\item If $(\mathcal{X}',\mathcal{X}'_{0})$ is a given log resolution of
$(\mathcal{X},\mathcal{X}{}_{0})$ with $E_{i}$ denoting the reduced
components of $\mathcal{X}'_{0},$ then the following formula holds
\begin{equation}
q=\max_{i}\frac{m_{i}-1-c_{i}}{m_{i}}+\frac{1}{L^{n}}\sum_{i}c_{i}\mathcal{L}'^{n}\cdot E_{i},\label{eq:formula for q in thm}
\end{equation}
where $m_{i}$ and $c_{i}$ are the order of vanishing along $E_{i}$
of $\mathcal{X}'_{0}$ of $\pi'^{*}\tau$ and any given non-trivial
meromorphic (multi-)section $s'$ of $\mathcal{L}'+\mathcal{K}_{\mathcal{X}'/\C}\rightarrow\mathcal{X}',$
respectively, i.e. if $\Delta'$ denotes the zero-divisor of $s',$
then
\[
\mathcal{X}'_{0}=\sum_{i}m_{i}E_{i},\,\,\,\Delta'=\sum_{i}c_{i}E_{i}
\]

\item $q=0$ iff $\mathcal{X}$ is $\Q-$Gorenstein with $\mathcal{L}$
isomorphic to $-K_{\mathcal{X}/\C}$ and $\mathcal{X}_{0}$ is reduced
and its normalization has log canonical singularities.
\end{itemize}
\end{thm}
\begin{proof}
First observe that we may as well assume that $\phi^{t}$ is a weak
geodesic ray. Indeed, if $\psi^{t}$ is the ray corresponding to a
locally bounded metric $\psi$ on $\mathcal{L}$ then $\phi-\psi$
is uniformly bounded and hence $f(t):=f_{1}(t)-f_{2}(t):=v_{\phi}(e^{-t/2})-v_{\psi}(e^{-t/2})$
(compare formula \ref{eq:def of v phi}) is bounded as $t\rightarrow\infty.$
But since $f_{i}(t)$ is convex (by Lemma \ref{lem:(Berndtsson-Paun)-Let-})
the limit of $df_{i}(t)/dt$ as $t\rightarrow\infty$ exists (a priori
in $]0,\infty])$ and since $f(t)$ is bounded it follows that the
limits of $df_{i}(t)/dt$ coincide. Similarly, $g(t):=\mathcal{E}(\phi_{t})-\mathcal{E}(\psi_{t})$
is a difference of convex functions (compare the proof of Prop \ref{prop:The-Ding-metric is cont})
and hence the limits of $d\mathcal{E}(\phi_{t})/dt$ and $d\mathcal{E}(\psi_{t})/dt$
coincide and thus so do the limits of $d\mathcal{D}(\phi_{t})/dt$
and $d\mathcal{D}(\psi_{t})/dt.$ 

To simplify the notation we will in the rest of the proof assume that
$X$ is smooth so that $D'=0,$ but the proof in the general case
is essentially the same. Fix a trivializing section $s$ of $\pi'_{*}(\mathcal{L}'+K_{\mathcal{X}'/\C})\rightarrow\C.$
The section $s$ induces an isomorphism between $\mathcal{L}$ and
$-K_{\mathcal{X}^{*}/\C^{*}}$ over $\mathcal{X}^{*}.$ In fact, since
the formula for $DF(\mathcal{X},\mathcal{L})$ is invariant under
an overall twist of the action $\rho$ on $\mathcal{L}$ we may as
well assume that $s$ is an invariant section and hence, using the
notation in the previous lemma $\deg\pi'_{*}(\bar{\mathcal{L}'}+K_{\bar{\mathcal{X}'}/\P^{1}})=0.$
We also fix a trivializing (mulit-)section $\sigma_{1}$ of the $\Q-$line
$-\frac{1}{L^{n}(n+1)}\left\langle \mathcal{L}',...,\mathcal{L}'\right\rangle _{|\tau=1}.$
By Lemma \ref{lem:Hilbert-Mumford} 
\[
w(\delta')=-\lim_{t\rightarrow\infty}\frac{d}{dt}\log\left\Vert \rho(\tau)S_{1}\right\Vert _{\Phi'}^{2}+l_{0},
\]
 where $S_{1}=\sigma_{1}\otimes s{}_{1}\in\delta'_{|\tau=1}$ and
$l_{0}$ is the Lelong number of the metric $\Phi'$ on $\delta'.$
Now, $\left\Vert \rho(\tau)S_{1}\right\Vert _{\Phi'}^{2}=\left\Vert S_{1}\right\Vert _{\rho(\tau)^{*}\Phi'_{|\mathcal{X}_{\tau}}}^{2}$
and hence setting $\phi^{t}=\rho(\tau)^{*}\phi_{\tau}$ and fixing
a metric $\psi$ on $-K_{X}$ we can write 
\[
-\log\left\Vert \rho(\tau)S_{1}\right\Vert _{\Phi'}^{2}+\log\left\Vert \sigma_{1}\right\Vert _{\psi_{D}}^{2}=-\frac{1}{L^{n}}\mathcal{E}(\phi_{t},\psi)-\log\int_{X}e^{-\phi_{t}}:=\mathcal{D}(\phi_{t})
\]
 using the previous identifications and the change of metrics formula
for the Deligne pairing \ref{eq:change of metric formula as energy}.
Now, using $DF(\mathcal{X},\mathcal{L})=w_{0}(\eta')$ and the decomposition
formula in Lemma \ref{eq:eta as sum of delta' + correction} together
with formula \ref{eq:intersection in terms of e} and Lemma \ref{lem:Hilbert-Mumford}
gives 
\begin{equation}
DF(\mathcal{X},\mathcal{L})=\lim_{t\rightarrow\infty}\frac{d}{dt}\mathcal{D}(\phi^{t})+q,\,\,\,\, q:=l_{0}+\frac{1}{L^{n}}\sum_{i}c_{i}\mathcal{L}'^{n}\cdot E_{i},\label{eq:q in proof of formula for df}
\end{equation}
 where $c_{i}$ is the order of vanishing of $s$ along $E_{i},$
when $s$ is viewed as a global holomorpic section of $\mathcal{L}'+K_{\mathcal{X}'/\C}\rightarrow\mathcal{X}'$.
Moreover, by the trivializing assumption on $s$ the numbers $c_{i}$
above coincide with thouse appearing in the formula for $l_{0}$ in
Prop \ref{prop:lelong of direct image} and hence formula \ref{eq:formula for q in thm}
follows. Note that both terms appearing in the definition of $q$
above are non-negative and hence $q\geq0.$ Indeed, by Prop \ref{prop:The-Ding-metric has pos prop}
$l_{0}\geq0$ and the non-negativity of the second terms follows directly
from the definitions giving that $c_{i}\geq0$ and $\mathcal{L}'$
is semi-ample. Next note that a general meromorphic section of $\mathcal{L}'+K_{\mathcal{X}'/\C}\rightarrow\mathcal{X}'$
may be written as $f(\tau)s'$ for $f(\tau)$ a meromorphic function,
whose vanishing (or pole) order at $\tau=0$ we denote by $m.$ Since
the formula for $q$ is invariant under $c_{i}\rightarrow c_{i}+m$
the case of a general section thus follows. Accordingly, in the rest
of the proof we will take $c_{i}$ to be the non-negative numbers
determined by the globally trivializing section $s'.$ 

Now, by formula \ref{eq:q in proof of formula for df} $q=0$ iff
the follows condition holds: $l_{0}=0$ and $c_{i}=0$ for all index
$i$ in the set $I$ defined by the condition $\mathcal{L}'^{n}\cdot E_{i}>0.$
But by formula \ref{eq:formula for q in thm} the latter condition
holds iff $m_{i}=1$ and $c_{i}=0$ for any $i\in I,$ i.e. any $i$
such that $E_{i}$ is not $p-$exceptional for the log resolution
$p$ (since $\mathcal{L}$ is assumed relative ample). Since $\mathcal{X}$
is normal we may, by Hironaka's theorem, take $p$ to be an isomorphism
on $p^{-1}(\mathcal{X}-\mathcal{Z}),$ where $\mathcal{Z}$ is a subvariety
of codimension at least two (containing the singular locus of $\mathcal{X}).$
Hence, if $q=0$ then $\mathcal{X}_{0}$ is reduced at any point in
$\mathcal{X}-\mathcal{Z}$ (using that, by the previous argument,
$m_{i}=1$ for any non $p-$exceptional $E_{i}$$).$ In other words,
$q=0$ implies that the central fiber $\mathcal{X}_{0},$ viewed as
a divisor on the normal variety $\mathcal{X}$, is reduced. But then,
since $q=0$ also implies that $c_{i}=0$ for any $i\in I$ it also
follows that $\mathcal{L}$ is isomorphic to $-K_{\mathcal{X}/\C}$
on $\mathcal{X}-\mathcal{Z}$ and since the codimension of $\mathcal{X}-\mathcal{Z}$
is at least two $\mathcal{L}$ is thus the unique extension of $-K_{\mathcal{X}/\C}$
from the regular locus of $\mathcal{X},$ which, by definition, means
that $\mathcal{X}$ is $\Q-$Gorenstein. Conversely, if $\mathcal{X}_{0}$
is reduced and $\mathcal{X}$ is $\Q-$Gorenstein, it follows from
Prop \ref{prop:lelong of direct image} that $l_{0}=0$ and hence
$q=0.$ 
\end{proof}

\subsection{\label{sub:An-alternative-proof}Conclusion of the proof of Theorem
\ref{thm:k-poly intro} }

Given a test configuration $(\mathcal{X},\mathcal{L})$ for $(X,-K_{X})$
Theorem \ref{thm:df=00003Dding} gives that for any weak geodesic
$\phi^{t}$ ray emanating from any given metric on $\mathcal{L}$
which is associated to $(\mathcal{X},\mathcal{L})$ we have 
\[
DF(\mathcal{X},\mathcal{L})=\lim_{t\rightarrow\infty}\frac{d}{dt}\mathcal{D}(\phi^{t})+q,\,\,\,\,\, q\geq0
\]
Next, by the convexity of $\mathcal{D}(\phi^{t})$ the limit in the
right hand side above is bounded from below by the right derivative
$\frac{d}{dt}\mathcal{D}(\phi_{t})_{|t=0^{+}}$ which, by formula
\ref{eq:ke is almost critical pt}, is non-negative if $\phi^{0}$
is taken as a Kähler-Einstein metric. Thus $DF(\mathcal{X},\mathcal{L})\geq0$
and if $DF(\mathcal{X},\mathcal{L})=0$ then it must, since $q\geq0,$
be that $\lim_{t\rightarrow\infty}\frac{d}{dt}\mathcal{D}(\phi_{t})=0$
and hence $\mathcal{D}(\phi_{t})$ is affine so that the second point
in Prop \ref{prop:The-Ding-metric has pos prop} implies that $\mathcal{X}$
is isomorphic to a product test configuration.

\section{Ramifications and applications}

\subsection{\label{sub:An-analog-of donalds conj} An analog of Donaldson's conjecture
about geodesic stability}

Combining the results above with the very recent existence result
in \cite{c-d-s} one arrives at the following analog of a conjecture
of Donaldson \cite{do00} (see \cite{c} for partial results about
Donaldson's original conjecture):
\begin{thm}
\label{thm:don conj for ding}Let $X$ be a Fano manifold. Then precisely
one of the following two alternatives holds:
\begin{enumerate}
\item $X$ admits a Kähler-Einstein metric
\item For any given $\phi^{0}\in\mathcal{H}_{b}(-K_{X})$ there exists a
weak geodesic ray $\phi^{t}$ in $\mathcal{H}_{b}(-K_{X})$ emenating
from $\phi^{0}$ such that the Ding functional $\mathcal{D}(\phi^{t})$
is strictly descreasing for sufficently large times.
\end{enumerate}
\end{thm}
\begin{proof}
If the first alternative holds, then it follows immediately, by the
convexity of $\mathcal{D}(\phi^{t})$ (just as in the proof of Theorem
\ref{thm:k-poly intro}) that the second alternative cannot hold.
Now assume that the first alternative does not hold. Then, by the
results in \cite{c-d-s} $X$ is not K-polystable along special test
configurations, i.e. there exists a special test configuration $\mathcal{X}$
such that one of the following alternatives hold $(a)$ $DF(\mathcal{X})<0$
or $(b)$ $DF(\mathcal{X})=0,$ but $\mathcal{X}$ is a not a product
test configuration. Now, by Theorem \ref{thm:df=00003Dding} $DF(\mathcal{X})$
is the large time limit of $d\mathcal{D}(\phi^{t})/dt$ where $\phi^{t}$
is any weak geodesic ray attached to $\mathcal{X}.$ Assuming, to
get a contradiction, that alternative above $2$ does not hold, there
exists a sequence of $t_{i}\rightarrow\infty$ such that $d\mathcal{D}(\phi^{t})/dt\geq0.$
By the convexity of $\mathcal{D}(\phi^{t})$ this means that there
exists a $T>0$ such that $d\mathcal{D}(\phi^{t})/dt\geq0$ on $[T,\infty[.$
In particular, $DF(\mathcal{X})\geq0$ and hence it must be that alternative
$(b)$ holds, i.e. $DF(\mathcal{X})=0$ and thus by, convexity, $\mathcal{D}(\phi^{t})$
is affine on $[T,\infty[.$ But then it follows from Prop \ref{prop:The-Ding-metric has pos prop}
that $\mathcal{X}$ is a product test configuration, which contradicts
$(b).$
\end{proof}
In the original conjecture of Donaldson $(X,-K_{X})$ is replaced
by a general polarized manifold $(X,L)$ and the Ding functional with
the Mabuchi functional. Moreover, originally Donaldson's conjecture
asked for bona fide geodesic rays $\phi^{t}$ of smooth and stricly
positively curved metrics, but in view of the recent theory about
geodesics one would expect that the best regularity that one can hope
for is that $\omega^{t}:=dd^{c}\phi^{t}$ be locally bounded, if $\omega^{0}$
is a Kähler form. By the regularity results in \cite{p-s1b}, this
is indeed the case in Theorem \ref{thm:don conj for ding} above.
Finally it should be pointed out that a weaker version of Theorem
\ref{thm:don conj for ding} has independently been obtained in \cite{d-h},
where it is assumed that $X$ admits no holomorphic vector fields
and where the ``destabilizing'' weak geodesic ray $\phi^{t}$ appearing
in item $2$ is merely in a finite energy class. On the other hand
the proof in \cite{d-h} dos not rely on the results in \cite{c-d-s}
(but rather estimates along the Kähler-Ricci flow).

\subsection{\label{sub:Applications-to-bounds}Bounds on the Ricci potential
and Perelman's $\lambda-$entropy functional}

Let now $X$ be a Fano manifold and denote by $\mathcal{K}(X)$ the
space of all Kähler metrics $\omega$ in $c_{1}(X)$ (equivalently,
$\omega=dd^{c}\phi$ for some strictly positively curved metric $\phi$
on $-K_{X}).$ In this section we will use the normalization $V:=c_{1}(X)^{n}:=\int_{X}\omega^{n}.$
Recall that the Ricci potential $h_{\omega}$ is the function on $X$
defined by $dd^{c}h_{\omega}=\mbox{Ric }\omega-\omega$ together with
the normalization condition $\int e^{h_{\omega}}\omega^{n}/V=1,$
which in terms of the previous notation means that $h_{dd^{c}\phi}:=h_{\phi}:=-\log(\frac{(dd^{c}\phi)^{n}/V}{e^{-\phi}/\int e^{-\phi}}).$
Note in particular that 
\[
\left\Vert 1-e^{h_{\omega}}\right\Vert _{L^{1}(X,\omega)}=\left\Vert \frac{1}{V}(dd^{c}\phi)^{n}-\frac{e^{-\phi}}{\int e^{-\phi}}\right\Vert ,
\]
 where the norm in the right hand side is the total variation norm
on the space of absolutely continuous probability measures on $X.$

Next, let $(\mathcal{X},\mathcal{L})$ be a test configuration of
a polarized manifold $(X,L)$ and define its ``$L^{\infty}-$norm''
by 
\begin{equation}
\left\Vert (\mathcal{X},\mathcal{L})\right\Vert _{\infty}:=\left\Vert \frac{d\phi^{t}}{dt}_{|t=0}\right\Vert _{L^{\infty}(X)},\label{eq:def of l infty norm of test}
\end{equation}
 where $\phi^{t}$ is the (weak) geodesic determined by $\mathcal{X},$
emanating from any fixed reference metric $\phi^{0}\in\mathcal{H}(X,L).$
The point is that if $\left\Vert (\mathcal{X},\mathcal{L})\right\Vert _{\infty}\neq0$
then the\emph{ normalized Donaldson-Futaki invariant} $DF(\mathcal{X},\mathcal{L})/\left\Vert (\mathcal{X},\mathcal{L})\right\Vert _{\infty}$
is independent of base changes of $(\mathcal{X},\mathcal{L}),$ induced
by $\tau\rightarrow\tau^{m}$ (which correspond to reparametrizations
of $\phi^{t},$ induced by $t\mapsto mt$$).$ We will be relying
on the following lemma which is a special case of a very recent result
of Hisamoto \cite[Theorem 1.1]{hi}:
\begin{lem}
\label{lem:l infty norm}The number $\left\Vert (\mathcal{X},\mathcal{L})\right\Vert _{\infty}$
is well-defined, i.e. it is independent of $\phi_{0}.$ 
\end{lem}
Now we can prove the following theorem using a slight variant of the
proof of Theorem \ref{thm:DF=00003Dding intro}; the result can be
seen as an analog of Donaldson's lower bound on the Calabi functional
\cite{do2},
\begin{thm}
Let $X$ be a Fano manifold. Then 
\[
\inf_{\omega\in\mathcal{K}(X)}\left\Vert 1-e^{h_{\omega}}\right\Vert _{L^{1}(X,\omega)}\geq\sup_{(\mathcal{X},\mathcal{L})}\frac{-DF(\mathcal{X},\mathcal{L})}{\left\Vert (\mathcal{X},\mathcal{L})\right\Vert _{\infty}},
\]
 where $(\mathcal{X},\mathcal{L})$ ranges over all test configurations
$(\mathcal{X},\mathcal{L})$ such that $\left\Vert (\mathcal{X},\mathcal{L})\right\Vert _{\infty}\neq0.$
Moreover, if equality holds and the infimum is attained at some $\omega$
and the supremum is attained at $(\mathcal{X},\mathcal{L})$ (with
$\mathcal{X}$ normal), then $(\mathcal{X},\mathcal{L})$ is isomorphic
to a product test configuration. In particular, 
\[
\inf_{\omega\in\mathcal{K}(X)}\int h_{\omega}e^{h_{\omega}}\frac{\omega^{n}}{V}\geq\frac{1}{2}\sup_{(\mathcal{X},\mathcal{L})}\left(\frac{DF(\mathcal{X},\mathcal{L})}{\left\Vert (\mathcal{X},\mathcal{L})\right\Vert _{\infty}}\right)^{2}
\]
where the sup ranges over all destabilizing $(\mathcal{X},\mathcal{L})$
(i.e. $DF(\mathcal{X},\mathcal{L})>0)$ with the same same necessary
conditions for equality as before. In particular, if $X$ is K-unstable
then both infimums above are strictly positive.\end{thm}
\begin{proof}
Fix $(\mathcal{X},\mathcal{L})$ and $\phi^{0}\in\mathcal{H}(X,-K_{X})$
and denote by $\phi^{t}$ the corresponding (weak) geodesic. By convexity
of the Ding functional, combined with Theorem \ref{thm:DF=00003Dding intro}
(using that $q\geq0)$, we have 
\begin{equation}
\int_{X}\left(\frac{1}{V}(dd^{c}\phi_{0})^{n}-\frac{e^{-\phi_{0}}}{\int e^{-\phi_{0}}}\right)\frac{d\phi^{t}}{dt}\geq-\frac{d}{dt}\mathcal{D}(\phi^{t})_{t=0}\geq-\lim_{t\rightarrow\infty}\frac{d}{dt}\mathcal{D}(\phi^{t})\geq-DF(\mathcal{X},\mathcal{L}).\label{eq:proof of lower bound ricci potential}
\end{equation}
Applying Hölder's inequality with exponents $(q,p)=(1,\infty)$ thus
gives 
\begin{equation}
\left\Vert 1-e^{h_{\omega}}\right\Vert _{L^{1}(X,\omega)}\left\Vert \frac{d\phi^{t}}{dt}_{|t=0}\right\Vert _{L^{\infty}(X)}\geq-DF(\mathcal{X},\mathcal{L})\label{eq:holder}
\end{equation}
 and using the independence in the previous lemma then concludes the
proof of the first inequality of the Theorem. The second inequality
then follows immediately from the classical Csiszar-Kullback-Pinsker
inequality between the relative entropy and the total variation norm
\cite{c-k}. As for the equality case it follows, just as in the second
proof of Theorem \ref{thm:k-poly intro}, from the equality cases
in \ref{eq:proof of lower bound ricci potential}. Finally, if $X$
is $K-$unstable then there exists, by definition, a test configuration
such that $DF(\mathcal{X},\mathcal{L})>0$ and for any such test configuration
the inequality \ref{eq:holder} forces $\left\Vert (\mathcal{X},\mathcal{L})\right\Vert _{\infty}>0,$
which concludes the proof. 
\end{proof}
Recall that in the definition of a test configuration $(\mathcal{X},\mathcal{L})$
we have fixed an action $\rho$ on $\mathcal{L}$ and thus the norm
$\left\Vert (\mathcal{X},\mathcal{L})\right\Vert _{\infty}$ certainly
depends on $\rho.$ Indeed, twisting $\rho$ with a character of $\C^{*}$
shifts the tangent of $\phi^{t}$ with a constant. On the other hand,
$DF(\mathcal{X},\mathcal{L})$ is independent of such a twist and
hence the previous theorem still holds if we replace $\left\Vert (\mathcal{X},\mathcal{L})\right\Vert _{\infty}$
with its (smaller) normalized version obtained by replacing the $L^{\infty}(X)-$norm
in the definition \ref{eq:def of l infty norm of test} with the quotient
norm on the quotient space $L^{\infty}(X)/\R.$
\begin{rem}
\label{rem:Lemma--infty norm}As pointed out above Lemma \ref{lem:l infty norm}
is a special case of a general result of Hisamoto \cite{hi}, saying
that the measure $(\frac{d\phi^{t}}{dt})_{*}MA(\phi^{t})$ on $\R$
only depends on the test configuration $(\mathcal{X},\mathcal{L})$
and moreover is equal to the limiting normalized weight measures for
the $\C^{*}-$action, as conjectured by Witt-Nyström \cite{n}, who
settled the case of product test configurations. In particular, by
\cite{hi} all the $L^{p}-$norms $\left\Vert (\mathcal{X},\mathcal{L})\right\Vert _{p}$
of $\frac{d\phi^{t}}{dt}$ (integrating against $MA(\phi^{t}))$ only
depend on $(\mathcal{X},\mathcal{L})$ and coincide with the limits
of the corresponding $l^{p}-$norms of the weights $\{\lambda_{i}^{(k)}\}.$
In particular, letting $p\rightarrow\infty$ gives Lemma \ref{lem:l infty norm}.
Using this the proof of the previous theorem shows that the theorem
holds, more generally, when $\left\Vert (\mathcal{X},\mathcal{L})\right\Vert _{\infty}$
is replaced by $\left\Vert (\mathcal{X},\mathcal{L})\right\Vert _{p}$
for $p\in[1,\infty]$ and the $L^{1}-$norm with the corresponding
$L^{q}-$norm, where $q$ is the Young (Hölder) dual of $p.$ In fact,
as shown in \cite{b-h-n} a similar argument can be used to give a
new proof and extend to general $L^{p}-$norms Donaldson's lower bound
on the Calabi functional \cite{do2}. 
\end{rem}
Next, we recall that Perelman's W-functional \cite{pe}, when restricted
to the space all pairs $(\omega,f)$ such that $\omega$ as in the
space $\mathcal{K}(X)$ of all Kähler metrics $\omega$ in $c_{1}(X)$
and $f$ is a smooth function such  that $e^{-f}\omega^{n}$ has unit
mass, is given by 

\[
W(\omega,f):=\int_{X}(R_{\omega}+|\nabla f|^{2}+f)e^{-f}\omega^{n},
\]
 where $R_{\omega}$ is the scalar curvature of $\omega$ normalized
so that $\int R_{\omega}\omega=n$ for any $\omega\in c_{1}(X)$ (as
usual in the Kähler setting where the volume of the metrics is fixed
we have set Perelman's parameter $\tau$ to be equal to $1/2$ and
as in \cite{pe,ti-zhu,t-z--,he} we have subtraced the universal constant
$2n$ from Perelman's original definition). Then Perelman's $\lambda-$entropy
functional on $\mathcal{K}(X)$ is defined as 

\[
\lambda(\omega)=\inf_{f\in\mathcal{C}^{\infty}(X):\,\int e^{-f}\omega^{n}=1}W(\omega,f)
\]
 \cite{pe,ti-zhu,t-z--,he} and in particular $\lambda(\omega)\leq W(\omega,0)=nV.$ 
\begin{cor}
\label{cor:bound on lambda f}Let $X$ be an $n-$dimension Fano manifold.
Then \textup{
\[
\sup_{\omega\in\mathcal{K}(X)}\lambda(\omega)\leq nV-\frac{1}{2}\sup_{(\mathcal{X},\mathcal{L})}\left(\frac{DF(\mathcal{X},\mathcal{L})}{\left\Vert (\mathcal{X},\mathcal{L})\right\Vert _{\infty}}\right)^{2}
\]
}where $V=c_{1}(X)^{n}$ and $(\mathcal{X},\mathcal{L})$ ranges of
all destabilizing test configurations for $(X,-K_{X}).$ In particular,
if $X$ is K-unstable then $\lambda\leq nV-\epsilon$ for some positive
number $\epsilon.$\end{cor}
\begin{proof}
As explained in \cite{he} $\lambda(\omega)+\int h_{\omega}e^{h_{\omega}}\omega^{n}\leq nV$
(using $W(\omega,f)\leq W(\omega,-h_{\omega})$ and one integration
by parts) and hence the corollary follows immediately from the previous
theorem.\end{proof}
\begin{rem}
The previous inequality was inspired by the result in \cite{t-z--}
and its extension to general non-invariant Kähler metrics in \cite{he},
saying that 
\[
\sup_{\omega\in\mathcal{K}(X)}\lambda(\omega)\leq nV-\sup_{\xi\in\mbox{Lie}G}H(\xi),
\]
with equality if $X$ admits a Kähler-Ricci soliton, where $\mbox{Lie}G$
is the Lie algebra of a maximal compact subgroup in $\mbox{Aut\ensuremath{_{0}(X)}and }$$H$
is a certain concave functional on $\mbox{Lie}G,$ defined in \cite{t-z--}.
The proof in \cite{he} was based on the convexity of the functional
$v_{\phi^{t}},$ while we here use the convexity of the whole Ding
functional.
\end{rem}

\subsection{\label{sub:The-logarithmic-setting}The log Fano setting}

Let us briefly recall the more general setting of Kähler-Einstein
metrics on log Fano varieties \cite{bbegz} and log K-stability \cite{do-3,li1,o-s}.
In a nutshell, this setting is obtained from the previous one by replacing
the canonical line bundle $K_{X}$ with the \emph{log canonical line
bundle} $K_{(X,D)}:=K_{X}+D$ of a given\emph{ }log pair $(X,D).$
For example, $(X,D)$ is said to be a\emph{ (weak) log Fano variety}
if $-K_{(X,D)}$ is ample (nef and big). A \emph{log Kähler-Einstein
metric} $\omega$ associated to $(X,D)$ is, by definition, a current
$\omega$ in $c_{1}(-K_{(X.D}),$ defining a Kähler metric on $X_{reg}-D,$
with locally bounded potentials on $X$ and such that 
\[
\mbox{Ric }\ensuremath{\omega-[D]=\omega,}
\]
holds in the sense of currents, where $[D]$ denotes the current of
integration defined by $D.$ Equivalently \cite{bbegz}, this means
that $\omega$ is the curvature current of a locally bounded metric
$\phi_{KE}$ on $-K_{(X,D)}$ satisfying 
\[
(dd^{c}\phi_{KE})^{n}=Ce^{-(\phi_{KE}+\log|s_{D}|^{2})}
\]
 (for some constant $C)$ in the sense of pluripotential theory, where
we recall that $e^{-(\phi+\log|s_{D}|^{2})}$ denotes the measure
associated to a metric $\phi$ on $-K_{(X,D)};$ see section \ref{sub:canonical measures}.
The definitions are compatible with log resolutions (as in formula\ref{eq:pull-back of log can}).
Hence if $(X,D)$ is a weak log Fano variety, then so is $(X',D')$
and $\phi_{KE}$ is a log Kähler-Einstein metric for $(X,D)$ iff
$p^{*}\phi_{KE}$ is a log Kähler-Einstein metric for $(X',D').$ 
\begin{example}
\label{ex:edge-cone}If $(X,D)$ is log smooth and klt, i.e. $X$
is smooth and $D$ has simple normal crossings with coefficents $<1,$
it follows the regularity results in \cite{cgh,g-p} give that any
log Kähler-Einstein metric for $(X,D)$ has edge-cone singularities
along $D.$ Moreover, by \cite{j-m-r}, if the log Mabuchi functional
(or the log Ding functional) for $(X,D)$ is proper the metric even
admits a complete polyhomogenous expansion along $D$ (this is shown
in \cite{j-m-r} when $D$ has a singular component and the general
case is anounced in \cite{j-m-r}). However, one of the main points
of the approach in the present paper is that it only relies on very
weak regularity properties of the metric (the local boundedness of
$\phi_{KE})$ and that it is independent of any properness assumption.
\end{example}
The notion of K-stability has also been generalized to the log setting
(see \cite{do-3,li1,o-s}). Briefly, a test configuration for a log
Fano variety $(X,D)$ consists of a test configuration $(\mathcal{X},\mathcal{L})$
for $(X,L)$ where $L=-K_{(X,D)}.$ The $\C^{*}-$action, applied
to the support of $D$ in $\mathcal{X}_{1}$$,$ induces a $\C^{*}-$invariant
divisor $\mathcal{D}^{*}$ in $\mathcal{X}^{*}$ and we denote by
$\mathcal{D}$ its closure in $\mathcal{X}.$ The corresponding\emph{
log Donaldson-Futaki invariant} $DF(\mathcal{X},\mathcal{L};\mathcal{D})$
was defined in \cite{li1} (by imposing linearity it is enough to
consider the case when $D$ is reduced and irreducible). A direct
calculation reveals that, up to normalization, the definition in \cite{li1}
is equivalent to replacing the relative canonical divisor $K$ in
the intersection theoretic formula \ref{eq:df as intersection} with
the relative log canonical divisor $K+\mathcal{D},$ defined as a
Weil divisor (compare \cite{o-s}): 
\begin{equation}
DF(\mathcal{X},\mathcal{L};\mathcal{D})=\mu\mathcal{\bar{\mathcal{L}}}\cdot\mathcal{\bar{\mathcal{L}}}\cdots\mathcal{\bar{\mathcal{L}}}+(n+1)(K+\mathcal{D})\cdot\cdot\mathcal{\mathcal{\bar{\mathcal{L}}}}\cdots\mathcal{\mathcal{\bar{\mathcal{L}}}},,\label{eq:log df as intersection-1}
\end{equation}
where now $\mu=n(-(K_{X}+D))\cdot L^{n-1}/L^{n}.$ We can hence take
the latter formula as the definition of the invariant $DF(\mathcal{X},\mathcal{L};\mathcal{D}).$
Finally, $(X,D)$ is said to be \emph{log K-polystable} if, for any
test configuration, $DF(\mathcal{X},\mathcal{L};\mathcal{D})\geq0$
with equality iff the test configuration is equivariantly isomorphic
to a product test configuration. 
\begin{thm}
\label{thm:log k stab}Let $(X,D)$ be a log Fano variety admitting
a log Kähler-Einstein metric, where $D$ is an effective $\Q-$divisor
on $X.$ Then $(X,D)$ is log K-polystable.\end{thm}
\begin{proof}
The proof given for Theorem \ref{thm:k-poly intro} actually proves
Theorem \ref{thm:log k stab} as well, since, in the case when $X$
is singular, it uses a log resolution to replace $\mathcal{X}$ with
a pair $(\mathcal{X}'^{*},D^{*})$ and then uses the closure of the
divisor $D^{*}$ in $\mathcal{X}'.$ 
\end{proof}
The theorem thus confirms one direction of the log version of the
Yau-Tian-Donaldson conjecture formulated in \cite{li1}.

\section{\label{sec:Outlook-on-the}Outlook on the existence problem for Kähler-Einstein
metrics on $\Q-$Fano varieties}

Very recently the existence of a Kähler-Einstein metric on a K-polystable
Fano manifold $X$ was finally settled by Donaldson-Chen-Sun \cite{c-d-s}.
In this section we will briefly discuss how some of the results in
the present paper may be useful when considering the corresponding
existence problem on a singular Fano variety. We will follow Tian's
original program, which is based on Aubin's continuity method (see
the outline in \cite{ti2} and references therein), but using the
Ding functional as a replacement for the Mabuchi functional used in
\cite{ti2}. Howeer, it should be emphasized that the recent results
in \cite{c-d-s}  are based on a modification of Tian's program introduced
by Donaldson which involves Kähler-Einstein metrics with conical singularities
(obtained by replacing the smooth form $\eta$ in Aubin's equation
\ref{eq:aubins eq} below, with the current defined by a suitable
anti-canonical $\Q-$divisor $D).$ One motivation for using Aubin's
original method here is that Tian's conjecture on the partial $C^{0}-$estimate
(see H1 below) has now been proved along Aubin's continuity method
when $X$ is smooth (see \cite{sz} which builds on\cite{c-d-s})
and one can thus hope that it will eventually also be established
on singular Fano varieties. 

The main connection to the present paper stems from the following
immediate consequence of Theorem \ref{thm:df=00003Dding} applied
to a special test configuration, which, as will be explained below,
together with two the general hypotheses H1 and H2 gives the existence
of a Kähler-Einstein metric on a K-stable Fano manifold.
\begin{cor}
\label{cor:mab diverges}Let $X$ be a Fano variety with log terminal
singularities and $\mathcal{X}$ a special test configuration for
$X$ such that $DF(\mathcal{X})>0.$ Fix a smooth and positively curved
metric $\phi$ on $-K_{\mathcal{X}/\C}$ (more generally, local boundedness
is enough) and set $\phi^{t}:=\rho^{*}\phi_{\tau}.$ Then the Ding
functional $\mathcal{D}$ and the Mabuchi functional $\mathcal{M}$
both tend to infinity along $\phi^{t},$ as $t\rightarrow\infty.$\end{cor}
\begin{proof}
By Theorem \ref{thm:df=00003Dding} we have that $\lim_{t\rightarrow\infty}\frac{d}{dt}\mathcal{D}(\phi^{t})>0$
and hence $\mathcal{D}(\phi^{t})\rightarrow\infty$ which, by the
well-known inequality $\mathcal{M}\geq\mathcal{D},$ concludes the
proof.
\end{proof}
We recall that the Mabuchi functional $\mathcal{M}$ admits a natural
extension to the space $\mathcal{H}_{b}(-K_{X})$ taking values in
$]-\infty,\infty]$ such that $\mathcal{M}(\phi)$ is finite precisely
when the measure $MA(\phi)$ has finite pluricomplex energy and relative
entropy \cite{bbegz}. In particular, by the regularity results in
\cite{p-s1b}, $\mathcal{M}(\phi^{t})$ is finite, for any fixed $t,$
if the initial metric $\phi_{0}$ is smooth and hence under the assumption
in the previous corollary $\mathcal{M}(\phi^{t})\rightarrow\infty$
tends to infinity as $t\rightarrow\infty.$ See \cite{p-s,ch} for
related results in the case when the total space $\mathcal{X}$ is
assumed smooth. 

After recalling Tian's program in the smooth setting with an eye towards
the singular case we will comment on further complications arising
when considering the existence problem on general $\Q$- Fano varieties.

\subsection{The case of a smooth Fano variety $X$}

The starting point of Tian's program is the continuity equation 
\begin{equation}
\mbox{Ric \ensuremath{\omega_{t}=t\omega_{t}+(1-t)\eta,}}\label{eq:aubins eq}
\end{equation}
 where $\omega_{0}$ is a given Kähler metric of positive Ricci curvature
$\eta$ and $t\in[0,1]$ is a fixed parameter. Let $I$ be the set
of all $t$ such that a solution $\omega_{t}$ exists. As shown by
Aubin $I\cap[0,1[$ is open and non-empty and hence to prove the existence
of a Kähler-Einstein metrics, i.e. that $1\in I,$ it is enough to
show that $I$ is closed. More precisely, denoting by $T$ the boundary
of $I$ and taking $t_{i}\rightarrow T$ we can write $\omega_{t_{i}}=dd^{c}\phi_{t_{i}}$
for suitably normalized metrics $\phi_{t_{i}}$ on $-K_{X}$ (e.g.
satisfying $\sup_{X}(\phi_{t_{i}}-\phi_{0})=0)$ and to show that
$I$ is closed it is enough to establish the following $C^{0}-$estimate:
\begin{equation}
\sup_{X}\left|\phi_{t_{i}}-\phi_{0}\right|\leq C\label{eq:aubin c0 estimate}
\end{equation}
(then the higher order estimates follow using the Aubin-Yau $C^{2}-$estimate,
Evans-Krylov theory and elliptic boot strapping). Before continuing
we recall that the following properties hold along the continuity
path \ref{eq:aubins eq} for $t\geq t_{0}$ (for a fixed $t_{0}\in I):$
\begin{equation}
(i)\,\mbox{Ric \ensuremath{\omega_{t}\geq t_{0}\omega_{t},\,\,\,\,(ii)\,\,\,\mathcal{M}(\phi_{t})\leq C_{0},}}\label{eq:prop along cont path}
\end{equation}
where the first property follows immediately from the fact that $\eta\geq0$
and the second one from the well-known fact that $\mathcal{M}(\phi_{t})$
is decreasing in $t.$

In order to relate the desired $C^{0}-$estimate \ref{eq:aubin c0 estimate}
to algebraic properties of $X$ Tian proposed the following conjecture
stated in terms of the Bergman function $\rho_{\omega}^{(k)}(x),$
at level $k,$ associated to a Kähler metric $\omega$ on $X:$ 
\[
\mbox{\ensuremath{\rho_{\omega}^{(k)}(x)=\sum_{i=0}^{N_{k}}|s_{i}^{(\phi)}|^{2}e^{-k\phi},}}
\]
where $\phi$ is any metric on $-K_{X}$ with curvature form $\omega$
and $\{s_{i}^{(\phi)}\}$ is any base in $H^{0}(X,-kK_{X})$ which
is orthonormal with respect to the $L^{2}-$norm $\left\Vert \cdot\right\Vert _{k\phi}$
on $H^{0}(X,-kK_{X})$ determined by $\phi,$ i.e. $\left\Vert s\right\Vert _{k\phi}^{2}=\int_{X}|s|^{2}e^{-k\phi}\omega^{n}.$
\begin{itemize}
\item \textbf{(H1)} \label{(Tian's-partial-estimate).}(Tian's partial $C^{0}-$estimate).
Given $t_{0}\in]0,1],$ let $\mathcal{K}(X,t_{0})$ be the space of
all Kähler metrics $\omega$ in $c_{1}(X)$ such that Ric$\omega\geq t_{0}\omega.$
Then there exists a $k>0$ and $\delta>0$ such that $kL$ is very
ample and for any $\omega\in\mathcal{K}(X,t_{0}),$ 
\[
\inf_{X}\mbox{\ensuremath{\rho_{\omega}^{(k)}(x)\geq\delta}}
\]

\end{itemize}
(more precisely, the conjecture says that $k$ can be chosen arbitrarily
large). If the previous conjecture holds then, as follows immediately
from the definition of $\rho_{\omega}^{(k)},$ the desired $C^{0}-$estimate
holds \ref{eq:aubin c0 estimate} iff 
\begin{equation}
\sup_{X}\left|\phi_{t_{i}}^{(k)}-\phi_{0}\right|\leq C\label{eq:tians partial c0 estimate}
\end{equation}
where now $\phi_{t_{i}}^{(k)}$ is the Bergman metric at level $k$
determined by $\phi_{t_{j}},$ i.e. $\phi_{t_{j}}^{(k)}=\frac{1}{k}\log\sum_{i=0}^{N_{k}}|s_{i}^{(\phi_{t_{j}})}|^{2}.$
In other words: $\phi_{t_{i}}^{(k)}$ is the scaled pull-back of the
Fubini-Study metric $\phi_{FS}$ on $\mathcal{O}(1)\rightarrow\P^{N_{k}}$
under the Kodaira map $F_{j}$ determined by $\phi_{t_{j}}:$ 
\[
F_{j}:\, X\rightarrow\P^{N_{k}},\,\,\,\,\,\phi_{t_{i}}^{(k)}=F_{j}^{*}\phi_{FS}/k,\,\,\,\, F_{j}(X):=V_{j}
\]
 i.e. $F_{i}(x)=[s_{0}(x):s_{1}(x):\cdots s_{N_{k}}(x)],$ where now
$(s_{i})$ is a fixed base, which is orthonormal with respect to the
$L^{2}-$norm determined by $\phi_{t_{j}}$ (strictly speaking, due
to the choice of base $V_{i}$ is only determined modulo action of
the the unitary group $U(N_{k}+1),$ but since this group is compact
this fact will be immaterial in the following). After passing to a
subsequence we may assume that the projective subvariety $V_{j}:=F_{j}(X)\subset\P^{N_{k}},$
converges, in the sense of cycles, to an algebraic cycle $V_{\infty}$
in $\P^{N_{k}}.$ It was indicated by Tian \cite{ti2} that the validity
of the previous conjecture would imply that the cycle $V_{\infty}$
is reduced, irreducible and even defines a\emph{ normal} variety.
More precisely, we will make the following 
\begin{itemize}
\item \textbf{(H2)} $V_{\infty}$ is normal with log terminal singularities
and there is a one parameter subgroup $\rho:\C^{*}\rightarrow GL(N_{k}+1,\C)$
such that 
\[
\sup_{X}\left|\phi_{t_{j}}^{(k)}-\rho(\tau_{i})^{*}\phi_{FS}\right|\leq C
\]
where $\rho(\tau_{i})V_{0}$ also converges (in the corresponding
Hilbert scheme) to the normal variety $V_{\infty}$ 
\end{itemize}
Then, by standard properties of the Hilbert scheme $\rho$ determines
a (special) test configuration $(\mathcal{X},\mathcal{L})$ with central
fiber $V_{\infty}.$ In fact, as explained in \cite{c-d-s} the existence
of $\rho$ in the case of Donaldson's contintuity method follows from
the reductivity of the automorphism group of $(V_{\infty},D_{\infty})$
established in \cite{c-d-s}, where $D_{\infty}$ is a divisor on
$V_{\infty}$ induced by $\eta=[D].$ Now, assuming that $X$ is K-stable
(for simplicity we consider the case when $X$ admits no non-trivial
holomorphic vector fields, but the general argument is similar) we
have that $DF(\mathcal{X},\mathcal{L})\geq0$ with equality iff $(\mathcal{X},\mathcal{L})$
is equivariantly isomorphic to a product test-configuration (recall
that the total space $\mathcal{X}$ here is automatically normal and
even $\Q-$Gorenstein, by Lemma \ref{lem:char of special test}).
In the latter case, $\rho(\tau)^{*}\phi_{FS}-\phi$ is trivially bounded
and hence the desired $C^{0}-$estimate \ref{eq:aubin c0 estimate}
then holds, showing that $X$ indeed admits a Kähler-Einstein metric.
The main issue is thus to exclude the case of $DF(\mathcal{X},\mathcal{L})>0$
and this is where Cor \ref{cor:mab diverges} enters into the picture.
Thus, assuming the validity of H1 and H2 above we deduce from Cor
\ref{cor:mab diverges} (with $\phi$ the restriction of the Fubini-Study
metric) that if the second alternative $DF(\mathcal{X},\mathcal{L})>0$
holds, then the Ding functional $\mathcal{D}$ tends to infinity along
$\rho(\tau_{i})^{*}\phi_{FS}$ and hence it is unbounded from above
along $\phi_{t_{j}}^{(k)}.$ But this implies that $\mathcal{D}$
is also unbounded along the original sequence $\phi_{t_{j}}$ (also
using that if $|\psi-\psi'|\leq C$ then $|\mathcal{D}(\psi)-\mathcal{D}(\psi')|\leq2C,$
as follows immediately from the definition \ref{eq:def of ding functional}).
But, since $\mathcal{D}\leq\mathcal{M},$ this contradicts the property
$(ii)$ in formula \ref{eq:prop along cont path} hence it must be
that the first alternative, $DF(\mathcal{X},\mathcal{L})=0,$ holds
and thus $X$ admits a Kähler-Einstein metric, as desired.

\subsection{Towards the case of $\Q-$Fano varieties}

Let us finally discuss some of the new complications that arise when
trying to generalize Tian's program to the case of singular K-polystable
Fano varieties $X$ (by \cite{od} such a Fano variety $X$ automatically
has log terminal singularities). Taking $\eta$ to be a smooth semi-positive
form in $c_{1}(X)$ the continuity equations \ref{eq:aubins eq} are
defined as before and, by the results in \cite{bbegz}, the set $I$
is still non-empty, i.e. $T>0$ (using the positivity of the alpha
invariant of $X).$ The solutions $\omega_{t}$ define Kähler forms
on $X_{reg}$ with volume $c_{1}(X)^{n}/n!.$ Next, we note that Tian's
conjecture admits a natural generalization to general $\Q-$Fano varieties
if one uses the notion of (singular) Ricci curvature appearing in
\cite{bbegz} (and similarly for general log Fano varieties $(X,D)).$
However, one new difficulty that arises is the\emph{ openness }of
$I.$ From the point of view of the implicit function theorem the
problem is to find appropriate Banach spaces, encoding the singularities
of $X$ (the uniqueness of solutions to the formally linearized version
of equation \ref{eq:aubins eq}, for $t\in]0,1[,$ follows from the
results in \cite{bbegz}). On the other hand, another approach could
be to use the following lemma, where the properness refers to the
exhaustion function defined by the $J-$functional (see \cite{bbegz}
for the singular case).
\begin{lem}
The set $I$ is open iff the twisted Ding (Mabuchi) functional $\mathcal{D}_{t}$
(associated to the twisting form $(1-t)\eta)$ is proper for any $t\in I.$\end{lem}
\begin{proof}
If $\mathcal{D}_{t}$ is proper, then it follows from the results
in \cite{bbegz} that a solution $\omega_{t}$ exists. Conversely,
if a solution $\omega_{t}$ exists and $I$ is open, i.e. solutions
$\omega_{t+\delta}$ exist for $\delta$ sufficiently small, then
it follows from the convexity of $\mathcal{D}_{t+\delta}$ along weak
geodesics that $\mathcal{D}_{t+\delta}\geq C.$ But since $\delta$
may be taken to be positive this implies that $\mathcal{D}_{t}$ is
proper (and even coercive; compare \cite{bbegz}). 
\end{proof}
Note that in the case $n=2$ it is a basic fact that a projective
variety $X$ has log terminal singularities iff it has quotient singularities
(defining an orbifold structure on $X)$ and hence the two-dimensional
Fano varieties are precisely the orbifold Del Pezzo surfaces. In the
general Fano orbifold case, if one takes $\eta$ to be an orbifold
Kähler metric, the usual implicit function theorem applies to give
that $I$ above is indeed open. For the case of K-polystable Del Pezzo
surfaces with canonical singularities (i.e. ADE singularities) the
existence of Kähler-Einstein metrics was established very recently
in \cite{od-s-s}, using a different method, thus generalizing the
case of smooth Del Pezzo surfaces settled by Tian \cite{ti2-1}.


\begin{thebibliography}{10}
\bibitem{pe}Perelman, G. The entropy formula for the Ricci flow and
its geometric applications, Arxiv/0211159.
\end{thebibliography}
\end{document}